\documentclass[twocolumn]{svjour3} 

\usepackage{amssymb}
\usepackage{amsmath}
\usepackage{hyperref}
\usepackage{tikz}
\usepackage{verbatim}
\usepackage{color}

\newcommand{\usecolor}[2]{{\color{#1} #2}}

\newcommand{\blue}[1]{\usecolor{blue}{#1}}

\usetikzlibrary{arrows}

\newtheorem{defi}{Definition}
\newtheorem{lem}{Lemma}
\newtheorem{pro}{Proposition}
\newtheorem{teo}{Theorem}
\newtheorem{cor}{Corollary}
\newtheorem{rem}{Remark}

\newcommand{\alg}[1]{{\ensuremath{\boldsymbol{ #1}}}}

\newcommand{\conj}{\ensuremath{\mathbin{\&}}}

\def\pieceend{\hfill$\dashv$\endtrivlist}
\newenvironment{proofclaim}{\trivlist\item[\hskip
       \labelsep{\it Proof of the claim:\/}]\ignorespaces }{\hfill\pieceend \endtrivlist}

\newenvironment{proof2}{\trivlist\item[\hskip
       \labelsep{\it Proof \/}]\ignorespaces }{\hfill$\Box$ \endtrivlist}

\begin{document}

\title{Expanding FL$_{ew}$ with a Boolean connective}

\author{Rodolfo C. Ertola-Biraben \and Francesc Esteva \and Llu\'is Godo}

\date{}
\institute{
Rodolfo C. Ertola-Biraben \at State University of Campinas, Rua S\'ergio Buarque de Holanda 251, 13083-859 Campinas, SP, Brazil \\  
\email{rcertola@cle.unicamp.br}
\and
Francesc Esteva, Llu\'is Godo \at  Artificial Intelligence Research Institute - CSIC, Campus de la UAB s/n, 08193 Bellaterra, Catalonia, Spain \\ 
\email{\{esteva, godo\}@iiia.csic.es}
}

\maketitle

\thispagestyle{empty}

\begin{abstract}
We expand FL$_{ew}$ with a unary connective whose algebraic counterpart is the operation that gives 
the greatest complemented element below a given argument. 
We prove that the expanded logic is conservative and has the Finite Model Property. 
We also prove that the corresponding expansion of the class of residuated lattices is an equational class. 
\end{abstract}

\section*{ACKNOWLEDGEMENT}

This paper was published in Soft Computing (Springer-Verlag) on 23 July 2016. 
The final publication is available at http://dx.doi.org/10.1007/s00500-016-2275-y

\section{Introduction}

In this paper we study the expansion of the substructural logic FL$_{ew}$, 
i.e.\ Full Lambek calculus with exchange and weakening, with a unary connective $B$  
whose intended algebraic semantics is as follows: given a bounded integral commutative residuated lattice 
(or residuated lattice for short) {\bf{A}}, 
$Ba$ is the maximum, if it exists, of the Boolean elements of the universe $A$ below $a$, 
which we call \emph{the greatest Boolean below a}, that is, 
$$Ba = \max\{b \in A: b \leq a \mbox{ and } b \mbox{ is Boolean}\}.$$
In fact, this operator is similar to the so-called Baaz-Monteiro $\Delta$ operator, 
very often used in the context of mathematical fuzzy logic systems that are semilinear expansions of MTL. 
Baaz \cite{Ba} studied it in connection with G\"odel logic while H\'ajek \cite{Ha} investigated $\Delta$ in BL logics in general, 
see also \cite[Chapter 2]{Handbook} for a more general perspective. 
Indeed, in such a context of semilinear logics, 
i.e. logics that are complete with respect to a class of linearly ordered algebras, 
the semantics of $\Delta$ is exactly the above one for $B$: in a linearly-ordered MTL-algebra, 
$\Delta a = 1$ if $a = 1$, and $\Delta a = 0$ otherwise, since the only Boolean elements in a chain are 1 and 0; 
moreover, from a logical point of view, $\Delta \varphi$ represents the weakest 
Boolean proposition implying $\varphi$. 

The operator $B$ can be also related to the join-complement operation $D$, also known as dual intuitionistic negation, 
already considered by Skolem \cite{Sk1} in the context of lattices with relative meet-complement, 
and later independently studied by e.g. Moisil \cite{Moi1} and Rauszer \cite{Rau} as well, 
the latter in the context of expansions of Heyting algebras. 
It turns out that the operation $\neg D$ and its iterations, where $\neg$ is the residual negation, has also very similar properties to $B$, 
and in some classes of residuated lattices they even coincide. 

In this paper we study the operator $B$ in the context of FL$_{ew}$ and axiomatize it.  
We show that the usual axiomatics of the $\Delta$ operator is actually too strong to capture the above intended semantics. 
In fact, the axiom $$\Delta(\varphi \lor \psi) \to (\Delta \varphi \lor \Delta \psi)$$ is not sound for $B$ over FL$_{ew}$ any longer.  
Thus, $B$ is a weaker operator than $\Delta$. 
However, as we will see, $B$ keeps most of the properties of $\Delta$. 
In particular, the expansion of FL$_{ew}$ with $B$ is conservative, 
its corresponding class of algebras is an equational class, 
and has the same kind of deduction theorem as $\Delta$. 
Also, $B$ may also be interesting as $\neg B$ has a paraconsistent behaviour. 
On the negative side, the expansion of a semilinear extension of FL$_{ew}$ with $B$  needs not to remain semilinear. 

The paper is structured as follows. 
In Section 2 we overview well-known facts about residuated lattices and its Boolean elements, as well as basic facts about the logic FL$_{ew}$. 
Sections 3 and 4 contain an algebraic study of the operator $B$. 
In particular, in Section 3 we study basic properties and show, among other things, 
that the class $\mathbb{RL}^B$ of residuated lattices expanded with $B$ is an equational class and state the modalities, while 
in Section 4 we compare $B$ with the mentioned $\Delta$ and with an operation using the join-complement $D$.
Finally, in Section 5 we focus on logical aspects, introducing the logic FL$_{ew}^B$, i.e. the expansion of  FL$_{ew}$ with the operator $B$, 
and show that is a conservative expansion and has the Finite Model Property, and hence it is decidable.
We conclude with some remarks and open problems. 

We give appropriate references. 
However, the paper is self-contained.

\section{Preliminaries}

\subsection{Residuated lattices and Boolean elements}

In this section we recall some properties of residuated lattices as well as of their Boolean elements that we will use in the following sections. 

Following \cite{GJKO},  a bounded, integral, commutative residuated lattice, or residuated lattice for short, 
is an algebra {\bf{A}}$=(A; \wedge, \vee, \cdot, \to, 0, 1)$ of type $(2, 2, 2, 2, 0, 0)$ such that: 

\begin{itemize}
\item [-] $(A;\wedge, \vee, 0, 1)$ is a bounded lattice with $0 \leq a \leq 1$, for all $a \in A$, 
\item [-] $(A;\cdot,1)$ is a commutative monoid (i.e. $\cdot$ is commutative, associative, with unit $1$), and
\item [-] $ \to$ is the residuum of $\cdot$, i.e.,\\
$a \cdot b \leq c$ iff $a \leq b \to c$, for all $a, b, c \in A$,
\end{itemize} 

\noindent where $\leq$ is the order given by the lattice structure.  A negation operator is defined as $\neg x = x\to 0$.

The class of residuated lattices will be denoted by $\mathbb{RL}$. 
It is well known that $\mathbb{RL}$ is an equational class 
and that it constitutes the algebraic semantics of the substructural logic FL$_{ew}$ (see Section \ref{Flew}).

\begin{example} \label{5}
In what follows we will have occasion to refer several times to the residuated lattice structure 
defined on the five-element lattice of Figure \ref{five-godel} by taking $\cdot = \land$ and $\to$ its residuum. 
With these operations, it actually becomes a five element G\"odel algebra, that is, 
a residuated lattice with $\cdot$ being idempotent and satisfying the pre-linearity law $(a \to b) \lor (b \to a) = 1$. 
\end{example}

\begin{figure} [ht]
\begin{center}

\begin{tikzpicture}
    \tikzstyle{every node}=[draw, circle, fill=white, minimum size=3pt, inner sep=0pt, label distance=1mm]
 
    \draw (0,0)		node (0) [label=right:$0$] {}    
    -- ++(90:1cm)	node (r) [label=right:$r$] {} 
    -- ++(135:1cm)	node (s) [label=left:$s$] {}
    -- ++(45:1cm)	node (1) [label=right:$1$] {}
    -- ++(315:1cm)	node (t) [label=right:$t$] {} 
    -- ++ (r); end
    \end{tikzpicture}
\end{center}
\caption{\label{five-godel} A five element G\"odel algebra} 
\end{figure}
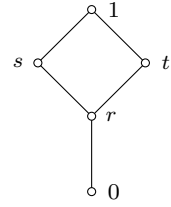

We omit the proof of the following well-known facts, see e.g. \cite{GJKO}. 

\begin{lem} \label{GRL} Let $\bf A \in \mathbb{RL}$. 
For any $a, b, c, d \in A$, the following properties hold: 

\vspace{5pt}

\emph{(i)} if $a \vee b = 1$, $a \leq c$, and $b \leq c$, then $c = 1$,

\emph{(ii)} if $a \vee b = 1$, $a\cdot c\leq d$, and $b\cdot c\leq d$, then $c \leq d$,

\emph{(iii)} if $a \leq b$, then $\neg b \leq \neg a$,

\emph{(iv)} $a \wedge \neg b \leq \neg(a \wedge b)$,

\emph{(v)} $a \cdot \neg b \leq \neg(a \to b)$,

\emph{(vi)} if $a \vee b = 1$, then $\neg a \leq b$,

\emph{(vii)} $a \leq \neg \neg a$.
\end{lem}

Special elements in a residuated lattice are those that behave as elements in a Boolean algebra. 

\begin{defi}
Let $\bf{A} \in \mathbb{RL}$. 
An element $a$ of its universe $A$ is called \emph{Boolean} or \emph{complemented} iff 
there is an element $b \in A$ such that $a \wedge b=0$ and $a \vee b=1$.
\end{defi}

In the rest of this section we state several properties of Boolean elements that will be useful in what follows. 
Even if most of them are folklore, we include proofs for all of them for the sake of being self-contained. 

An equivalent and simpler condition for an element to be Boolean is the following. 

\begin{lem} \label{BL}
An element $a$ in the universe of a residuated lattice is Boolean iff $a \vee \neg a = 1$.
\end{lem}

\begin{proof2}
 $\Rightarrow)$ Suppose there is an element $b$ such that $a \wedge b=0$ and $a \vee b=1$. 
 First, using that $a \wedge b=0$ and $a \cdot b \leq a \wedge b$, we have that $a \cdot b=0$. 
 So, $b \leq a \to 0$, i.e. $b \leq \neg a$. 
 Secondly, we have that $\neg a=\neg a \cdot 1= \neg a \cdot (a \vee b)=
 (\neg a \cdot a) \vee (\neg a \cdot b)=\neg a \cdot b$. 
 So, $\neg a=\neg a \cdot b$. 
 As $\neg a \cdot b \leq b$, it follows that $\neg a \leq b$. 
 So, $b=\neg a$. 
 As we have that $a \vee b=1$, it follows that $a \vee \neg a =1$.

\noindent $\Leftarrow)$ By hypothesis, we have (i) $a \vee \neg a=1$. 
It is enough to see that $a \wedge \neg a = 0$. 
As $a \cdot \neg a = 0$, it is enough to prove that $a \wedge \neg a \leq a\cdot \neg a$. 
We have that $a \wedge \neg a \leq \neg a$. 
So, by monotonicity of $\cdot$, we have (ii) $a\cdot (a\wedge \neg a) \leq a\cdot \neg a$. 
We also have that $a \wedge \neg a \leq a$. 
So, again by monotonicity of $\cdot$, we have $\neg a \cdot (a\wedge \neg a) \leq \neg a \cdot a = a \cdot \neg a$. 
So, it follows (iii) $\neg a \cdot (a \wedge \neg a) \leq a\cdot \neg a$. 
Now, using Lemma \ref{GRL}(ii) with (i), (ii), and (iii), it follows that $a \wedge \neg a \leq a \cdot \neg a$. 
\end{proof2}

\begin{pro} \label{infmon}
Let $\bf{A} \in \mathbb{RL}$ and let $a$ be a Boolean element of its universe $A$. 
Then, for all $b, c, d \in A$ the following properties hold: 

\vspace{5pt}

\emph{(i)} $a\wedge b = a\cdot b$,

\emph{(ii)} $a\cdot a = a$,

\emph{(iii)} $a \wedge \neg a = 0$,

\emph{(iv)} $a \wedge(b \vee c) = (a \wedge b)\vee (a \wedge c)$,

\emph{(v)} $\neg \neg a = a$,

\emph{(vi)} $a \to b = \neg a \vee b$, 

\emph{(vii)} $0 = \neg(a \vee \neg a)$,

\emph{(viii)} if $a \leq b \vee c$, $a \wedge b \leq d$, $a \wedge c \leq d$, then $a \leq d$, 

\emph{(ix)} if $b \vee c=1$, $a \wedge b \leq d$,  $a \wedge c \leq d$, then $a \leq d$, 

\emph{(x)} if $a \wedge b \leq c$, then $a \wedge \neg c \leq \neg b$,  

\emph{(xi)} if $a \vee \neg b=1$, then $b \leq a$. 
\end{pro}

\begin{proof2}
(i) Suppose that (i) $a\vee \neg a = 1$. It is enough to see that $a \wedge b \leq a \cdot b$. 
We have that $a \wedge b \leq b$. 
So, by monotonicity of $\cdot$, we have (ii) $a \cdot (a \wedge b) \leq a \cdot b$. 
We also have that $a \wedge b \leq a$. 
So, again by monotonicity of $\cdot$, we have $\neg a \cdot (a \wedge b) \leq \neg a \cdot a = 0$. 
So, it follows (iii) $\neg a \cdot (a \wedge b) \leq a \cdot b$. 
Now, using Lemma \ref{GRL}(ii) with (i), (ii), and (iii), it follows that $a \wedge b \leq a \cdot b$.

\smallskip

(ii) Using Part (i), we have $a \wedge a = a \cdot a$. Also, $a \wedge a = a$. So, $a \cdot a = a$.

\smallskip

(iii) Using Part(i), we have $a \wedge \neg a = a \cdot \neg a$. Also, $a \cdot \neg a = 0$. So, $a \wedge \neg a = 0$.

\smallskip

(iv) In a residuated lattice it holds that $a \cdot(b \vee c)=(a \cdot b) \vee (a \cdot c)$. 
Let $a$ be Boolean. 
Then, using Part (i) three times, it follows that $a \wedge (b \vee c) = (a \wedge b) \vee (a \wedge c)$. 

\smallskip

(v) Suppose that (i) $a \vee \neg a = 1$. 
It is enough to see that $\neg \neg a \leq a$. We have that (ii) $a \cdot \neg \neg a \leq a$. 
Also, (iii) $\neg a \cdot \neg \neg a \leq a$, as $\neg a \cdot \neg \neg a = 0$. 
So, using Lemma \ref{GRL}(ii) with (i), (ii), and (iii), $\neg \neg a \leq a$.

\smallskip

(vi) It is enough to prove (i) $(\neg a \vee b) \cdot a \leq b$ and (ii) if $x \cdot a \leq b$, then $x \leq \neg a \vee b$. 
To see (i), note that $\neg a \cdot a \leq b$ and $b \cdot a \leq b$, whence $(\neg a \cdot a) \vee (b \cdot a) \leq b$. 
So, using distributivity of $\cdot$ relative to $\vee$, we get (i). 
To see (ii), suppose $x \cdot a \leq b$. 
Then, $x \leq a \to b$. 
In order to get $x \leq \neg a \vee b$, it is enough to derive (iii) $a \to b \leq \neg a \vee b$ and use transitivity of $\leq$. 
To get (iii), let us use Lemma \ref{GRL}(ii). 
As $a$ is Boolean, we have $a \vee \neg a=1$. 
Now, $a \cdot (a\to b) \leq b \leq \neg a \vee b$. Also, $\neg a \cdot (a \to b) \leq \neg a \leq \neg a \vee b$. 
So, using Lemma \ref{GRL}(ii), we get (iii). 

\smallskip

(vii) As we have $\neg c \leq a \vee \neg a$, for any $c \in A$, then, using Lemma \ref{GRL}(iii) and Part (iv), 
we get $\neg(a \vee \neg a) \leq \neg \neg c \leq c$.

\smallskip

(viii) As $a \leq b \vee c$, we have $a \leq a \wedge (b \vee c)$. 
Now, using Part (iv), it follows that $a \leq (a \wedge b) \vee (a \wedge c)$. 
Now, as $a \wedge b \leq d$ and $a \wedge c \leq d$, we have, by a basic property of $\vee$, 
that $(a \wedge b) \vee (a \wedge c) \leq d$. 
So, by transitivity of $\leq$, $a \leq d$. 

\smallskip

(ix) As $b \vee c=1$, reason as in Part (viii).

\smallskip

(x) Suppose $a \wedge b \leq c$. 
Then $a \cdot b \leq c$. 
By monotonicity of $\cdot$, it follows that $(a \cdot b) \cdot \neg c \leq c \cdot \neg c = 0$. 
Then, as $\cdot$ is both associative and commutative, $(a \cdot \neg c) \cdot b \leq 0$. 
So, $a \cdot \neg c \leq \neg b$. 
Finally, using that $a$ is Boolean, we get $a \wedge \neg c \leq \neg b$. 

\smallskip

(xi) Let $a \vee \neg b=1$. Then, $b \cdot (a \vee \neg b)=b \cdot 1 = b$. 
By distributivity of $\cdot$ relative to $\vee$, it follows that $(b \cdot a) \vee (b \cdot \neg b) = b$. 
As $b \cdot \neg b=0$, we have that $b \cdot a=b$. 
So, as $a$ is Boolean, $b \wedge a = b$, i.e. $b \leq a$.
\end{proof2}

\begin{lem} \label{inflem}
Let $\bf{A} \in \mathbb{RL}$ and let $a$ and $b$ be Boolean elements of $A$. 
Then, 

\smallskip

\emph{(i)} $a \wedge b = a \cdot b =\neg (\neg a \vee \neg b)$,

\smallskip

\emph{(ii)} $(a\cdot b) \vee (\neg a \cdot b) \vee (a\cdot \neg b) \vee (\neg a\cdot \neg b) = 1$.

\end{lem}

\begin{proof2}
(i) Firstly, we have that $\neg a \leq \neg a \vee \neg b$. 
So, using Lemma \ref{GRL}(iii), it follows that $\neg(\neg a \vee \neg b) \leq \neg \neg a$. 
Now, using Proposition \ref{infmon}(v) and $\leq$-transitivity, we have that $\neg(\neg a \vee \neg b) \leq a$. 
Analogously, we get $\neg(\neg a \vee \neg b) \leq b$. 
Secondly, suppose $c \leq a$ and $c \leq b$, for $c \in A$. 
Then, using Lemma \ref{GRL}(iii) again, it follows that $\neg a \leq \neg c$ and $\neg b \leq \neg c$. 
So, $\neg a \vee \neg b \leq \neg c$. 
Then, using Lemma \ref{GRL}(iii) once again, $\neg \neg c \leq \neg(\neg a \vee \neg b)$. 
Now, using Proposition \ref{infmon}(v) and $\leq$-transitivity, we get $c \leq \neg(\neg a \vee \neg b)$. 

\smallskip

(ii) Suppose that $a \vee \neg a=b \vee \neg b=1$. 
Then, $1=(a \vee \neg a) \cdot(b \vee \neg b)=(a\cdot b)\vee (\neg a\cdot b) \vee (a \cdot \neg b) \vee (\neg a\cdot \neg b)$.
\end{proof2}

\begin{pro} \label{infmon2}
Let $\bf{A} \in \mathbb{RL}$ and let $a$ and $b$ be Boolean elements of $A$. 
Then, (i) $\neg a$, (ii) $a \vee b$, (iii) $a \wedge b=a \cdot b$, (iv) $a \to b$, (v) $0$, and (vi) $1$ are Boolean. 
\end{pro}

\begin{proof2}

\smallskip

(i) Suppose $a \vee \neg a=1$. 
Then, as $a \leq \neg \neg a$, we get $\neg \neg a \vee \neg a = 1$.

\smallskip

\noindent (ii) Use Lemma \ref{inflem} (ii) and see that

\smallskip

$a\cdot b \leq a \leq a \vee b \leq (a\vee b) \vee \neg(a\vee b)$,

\smallskip

$\neg a\cdot b \leq b \leq a \vee b \leq (a\vee b) \vee \neg(a\vee b)$,

\smallskip

$a\cdot \neg b \leq a \leq a \vee b \leq (a\vee b) \vee \neg(a\vee b)$, and

\smallskip

$\neg a\cdot \neg b \leq \neg a \wedge \neg b \leq \neg (a\vee b) \leq (a\vee b) \vee \neg(a\vee b)$.

\smallskip

So, $a\vee b$ is Boolean.

\medskip

\noindent (iii) Use Parts (i) and (ii), and Lemma \ref{inflem}(i).

\medskip

\noindent (iv) Use Parts (i) and (ii), and Proposition \ref{infmon}(vi).

\medskip

\noindent (v) Use Parts (i) and (ii), and Proposition \ref{infmon}(vii).

\medskip

\noindent (vi) Use the definition of Boolean element and the fact that $\neg 1=0$.
\end{proof2}

From Proposition \ref{infmon2} it easily follows that, in any residuated lattice $\bf A$, 
the set of its Boolean elements $B(A)=\{a\in A: a$ is Boolean$\}$ is the domain of a subalgebra of $\bf A$, 
which is in fact a Boolean algebra. 
Indeed, ${\bf{B(A)}}=(B(A); \wedge, \vee, \cdot, \to, 0, 1)$ is the greatest Boolean algebra contained in {\bf{A}}. 
{\bf{B(A)}} is called the Boolean skeleton or  the \emph{center} of {\bf{A}}.

\subsection{On the logic FL$_{ew}$ }\label{Flew}

The logics we are interested in are extensions or expansions of the logic FL$_{ew}$ described below.

\begin{defi}
The language of {\rm FL$_{ew}$} has four binary connectives, $\wedge $, $\vee$, $\cdot$, and $\rightarrow$, and 
two constants, $0$ and $1$. 

\noindent The axioms of {\rm FL$_{ew}$} are:

\vspace{5pt}

\begin{tabular}{ll}
\emph{(1)} & $(\varphi \to \psi )\to ((\psi \to \gamma) \to (\varphi \to \gamma))$, \\ 
\emph{(2)} & $(\gamma \to \varphi) \to ((\gamma \to \psi) \to (\gamma \to (\varphi \wedge \psi)))$, \\ 
\emph{(3)} & $(\varphi \wedge \psi ) \to \varphi $ and $(\varphi \wedge \psi) \to \psi$, \\ 
\emph{(4)} & $\varphi \to (\varphi \vee \psi)$ and $\psi \to (\varphi \vee \psi)$, \\ 
\emph{(5)} & $(\varphi \to \gamma)\to ((\psi \to \gamma)\to ((\varphi \vee \psi)\to \gamma))$, \\  
\emph{(6)} & $(\varphi \cdot \psi )\to (\psi \cdot \varphi)$, \\ 
\emph{(7)} & $(\varphi \cdot \psi )\rightarrow \varphi $, \\ 
\emph{(8)} & $(\varphi \to (\psi \to \gamma )) \to ((\varphi \cdot \psi ) \to \gamma )$, \\ 
\emph{(9)} & $((\varphi \cdot \psi ) \to \gamma )\to (\varphi \to (\psi \to \gamma))$, \\ 
\emph{(10)} & $0 \to \varphi$ and $\varphi \to 1$. \\
\end{tabular}

\vspace{5pt}

\noindent The only rule of {\rm FL$_{ew}$} is \emph{modus ponens}: 
\[ {\frac{\varphi \quad \varphi \to \psi }{\psi }}. \]
\end{defi}

\noindent We define $\lnot \varphi = \varphi \to 0$ and 
$\varphi \leftrightarrow \psi = (\varphi \to \psi )\wedge (\psi \to \varphi )$. 

\medskip

The following formulas and rules are derivable in FL$_{ew}$:

\begin{tabular}{ll}

(11) & $\displaystyle{\frac{\varphi \to \psi \qquad \psi \to \gamma} {\varphi \to \gamma}}$, \\ 

(12) & $\varphi \to (\psi \to \varphi)$, \\ 

(13) & $\varphi \to \varphi$, \\ 

(14) & $\varphi \to (\varphi \cdot 1)$ and $(\varphi \cdot 1) \to \varphi$, \\ 

(15) & $(\varphi \cdot (\psi \cdot \gamma)) \leftrightarrow ((\varphi \cdot \psi)\cdot \gamma)$ \\  

(16) & $(\varphi\cdot (\varphi \to \psi)) \to \psi$, \\ 

(17) & $(1 \to \varphi) \to \varphi$, $\varphi \to (1 \to \varphi)$, \\ 

(18) & $\displaystyle{\frac{\varphi \qquad \psi}{\varphi \wedge \psi}}$, \\  

(19) & $(\varphi \vee \neg \varphi) \to \neg (\varphi \wedge \neg \varphi)$. \\ 
\end{tabular}

\vspace{5pt}

\noindent Derivations for (11)-(19) are rather easy. Hence, they are left to the reader.

We will occasionally consider the following extensions of FL$_{ew}$.

\begin{defi} \mbox{}  Consider the following axiomatic extensions of {\rm FL$_{ew}$}: 
\begin{itemize}

\item  Intutitionistic logic {\rm IL} is {\rm FL$_{ew}$} plus the axiom 

\noindent {\bf{(Contr)}} $\qquad  \varphi \to (\varphi \cdot \varphi)$.

\smallskip

\item The logic {\rm MTL} is {\rm FL$_{ew}$} plus the axiom

\noindent {\bf{(Prel)}} $\qquad  (\varphi \to \psi ) \vee (\psi \to \varphi)$.
\smallskip

\item {\rm SMTL} logic is {\rm MTL}  plus the axiom

{\bf (PC)} $\qquad  \neg(\varphi \land \neg \varphi) $
\smallskip

\item {\rm WNM} logic is {\rm MTL}  plus the axiom

{\bf (WNM)} $\qquad  \neg(\varphi \& \psi) \lor (\varphi \land \psi \to \varphi \& \psi)$
\smallskip

\item {\rm NM} logic is {\rm WNM}  plus the axiom

{\bf (Inv)} $\qquad  \neg \neg \varphi \to \varphi $
\smallskip

\item {\rm BL} logic is {\rm MTL}  plus the axiom

{\bf (Div)} $\qquad   (\varphi \land \psi)  \to ( \varphi \& (\varphi \to \psi))$
\smallskip

\item {\rm Product} logic is {\rm BL}  plus (PC) and  the axiom

{\bf (C)} $\qquad \neg\varphi \lor ((\varphi \to \varphi \& \psi) \to \psi)$
\smallskip

\item  G\"odel logic {\rm G} is  {\rm BL} plus  {\bf (Contr)} 
\smallskip

\item  {\L}ukasiewicz  logic {\rm \L} is  {\rm BL} plus  {\bf (Inv)}.
\end{itemize}
\end{defi}

\begin{rem}
 Note that G\"odel logic {\rm G} arises also as {\rm MTL} plus {\bf (Contr)} or as {\rm IL} plus  {\bf (Prel)}.  
\end{rem}

In MTL and its extensions we may define $\phi \vee \psi :=((\phi \to \psi) \to \psi)\wedge ((\psi \to \phi) \to \phi)$. 
In BL and its extensions we may define $ \varphi \land \psi :=  \varphi \& (\varphi \to \psi)$. 
Moreover, in IL the formula $(\phi \cdot \psi)\leftrightarrow (\phi \wedge \psi)$ is derivable, 
i.e. connectives $\wedge$ and $\cdot$ coincide in IL. 

All these logics are algebraizable, 
and hence they are strongly complete with respect to their corresponding classes of algebras. 
Namely, FL$_{ew}$ is complete with respect to the variety  $\mathbb{RL}$ of residuated lattices, 
MTL is complete with respect to the variety  of pre-linear residuated lattices (MTL-algebras), and 
IL is complete with respect to the variety  of contractive residuated lattices (Heyting algebras). 
Moreover, all axiomatic extensions of MTL are semilinear logics, that is, 
they are strongly complete with respect to the corresponding class of linearly ordered algebras. 
For instance, G\"odel logic is complete with respect to the class of  linearly ordered Heyting algebras, or G\"odel chains.

\section{Residuated lattices enriched with $B$}

As explained in the previous section, 
the set of Boolean elements of a residuated lattice $\bf A$ forms a Boolean algebra denoted \emph{the center or Boolean skeleton of $\bf A$}. 
Cignoli and Monteiro considered Boolean elements in \L ukasiewicz algebras in \cite{Ci} and \cite{CM}. 
However, as far as we know, the operator defining the greatest Boolean element below, i.e. the operator $B$ studied in this paper, 
has not yet  been studied in the general context of residuated lattices. 
One relevant exception is the paper \cite{RZ}, 
where Reyes and Zolfaghari define modal operators $\Box$ and $\Diamond$ in the context of Bi-Heyting algebras 
that are shown to correspond respectively to the greatest and the smallest complemented element below and above, respectively. 
Thus, the $\Box$ operation coincides with $B$. 
In the cited paper, using dual negation (or join-complement) $D$, always in the context of Bi-Heyting algebras, 
the authors also study a family of modal operators $\Box_n$ and $\Diamond_n$, 
in a similar way to the one we shall employ in Section \ref{join-complement}.

We will be considering residuated lattices {\bf{A}} enriched with 
a unary operation $B$ such that, for all $a \in A$, $Ba$ is the greatest Boolean element below $a$, as defined in the Introduction. 
It is clear that $B$ can be characterized by the following three conditions, for $a$, $b$ in $A$: 

\medskip

\begin{tabular}{ll}
{\bf{(BE1)}} & $Ba \leq a$, \\ 
{\bf{(BE2)}} & $Ba \vee \neg Ba=1$, \\  
{\bf{(BI)}}  & if $b \leq a$ and $b \vee \neg b=1$, then $b \leq Ba$. 
\end{tabular}

\medskip

\noindent The class of residuated lattices with $B$ will be denoted by $\mathbb{RL}^B$. 
Namely, an $\mathbb{RL}^B$-algebra is an algebra ${\bf A} =(A; \wedge, \vee, \cdot, \to, B, 0, 1)$ such that 
$(A; \wedge, \vee, \cdot, \to, 0, 1)$ is a residuated lattice and $B$ satisfies the above three conditions. 

\medskip

First of all, note that $B$ is new, that is, $B$ is not expressible by a $\{ \wedge, \vee, \cdot, \to, 0 \}$-term. 
Indeed, for instance, in the G\"{o}del algebra ${\bf G}_2 \times {\bf G}_3$  
(the direct product of the two-element Boolean algebra with universe $\{0, 1\}$ 
and the three-element G\"odel algebra with universe $\{0, \frac{1}{2}, 1\}$) 
we have, for any $\{ \wedge, \vee, \cdot, \to 0 \}$-term $t$, 
that $ta \in \{0, a, 1\}$, where $a =(1, \frac{1}{2})$ is the join reducible coatom, 
while $Ba = (1, 0)$ is the join-irreducible atom, which does not belong to  $\{0, a, 1\}$. 

In the next proposition we see that all operations remain independent. 

\begin{pro}
The set of operators $\{\wedge, \vee, \cdot, \to, B, 0 \}$ is independent.
\end{pro}

\begin{proof2}
To see that $\wedge$ is independent of the rest take the distributive lattice in Figure \ref{five-godel} 
and define the monoidal operation $\cdot$ as $s \cdot t = 0$, $s \cdot s = s$, and $t \cdot t = t$, for coatoms $s$ and $t$. 
This operation has a corresponding residuum $\to$. 
Since the only Boolean elements are 0 and 1, the operator $B$ is defined as $B1 = 1$ and $Ba = 0$, for all $a \neq 1$. 
Then, note that the set $S = \{0, s, t, 1 \}$ is closed for $\vee$, $\cdot$, $\to$, $0$ and $B$, but $s \wedge t \notin S$.

\smallskip

\noindent To see that $\vee$ is independent of the rest take 
the algebra that results from inverting the lattice order in the algebra of Example \ref{5} and 
note that the set $S$ with bottom, both atoms $a_{1}$ and $a_{2}$, and top is closed for $\wedge$, $\cdot$, $\to$, $0$, and $B$, 
but $a_{1} \vee a_{2} \notin S$.

\smallskip

\noindent To see that $\cdot$ is independent of the rest take the four-element chain $0 < a < b < 1$ where 
$a \cdot b = a \cdot a = b \cdot b = a$, for the atom $a$ and the coatom $b$, and 
note that the set $S =\{0, b, 1\}$ is closed for $\wedge$, $\vee$, $\to$, $0$, and $B$, 
but $b \cdot b \notin S$.

\smallskip

\noindent To see that $\to$ is independent of the rest take the algebra of Example \ref{5} and 
note that the set $S = \{0, r, s, 1 \}$ is closed for 
$\wedge$, $\vee$, $\neg$, $0$, and $B$, but $s \to r \notin S$. 

\smallskip

\noindent To see that $0$ is independent of the rest take the two element Boolean algebra and 
note that the set $S = \{ 1 \}$ is closed for $\wedge$, $\vee$, $\cdot$, $\to$, and $B$, but $0 \notin S$.  
 
\smallskip

\noindent The independence of $B$ has already been considered.
\end{proof2}

\begin{lem} \label{Blem} Let ${\bf A} \in \mathbb{RL}^B$ and let $a \in A$. Then, 

 \emph{(i)} $Ba = a$ iff $a$ is Boolean, 
 
 \emph{(ii)} $Ba = 1$ iff $a = 1$, 
 
 \emph{(iii)} $BBa = Ba$.
\end{lem}

\begin{proof2}

(i) Suppose $Ba = a$. Using {\bf{(BE2)}} it follows that $a \vee \neg a = 1$. 
For the other conditional, suppose $a \vee \neg a = 1$. 
Then, as $a \leq a$, using {\bf{(BI)}} it follows that $a \leq Ba$. 
The other inequality follows by {\bf{(BE2)}}. 

\smallskip

(ii) Suppose $Ba = 1$. Using {\bf{(BE1)}}, it follows that $1 \leq a$, i.e. $a = 1$. 
For the other conditional, suppose $a=1$. Then, $a \leq 1$. 
Using {\bf{(BI)}} and the fact that $1$ is Boolean (see Proposition \ref{infmon2}(vi)), it follows that $1 \leq Ba$.

\smallskip

(iii) Considering {\bf{(BE1)}}, it is enough to see that $Ba \leq BBa$, which follows using {\bf{(BI)}} and {\bf{(BE2)}}. 
\end{proof2}

We also have the following properties.

\begin{lem} \label{LB}
Let ${\bf A} \in \mathbb{RL}^B$ and let $a, b \in A$. Then, 

\vspace{5pt}

\begin{tabular}{ll}

\emph{(i)} & $B$-Monotonicity: if $a \leq b$, then $Ba \leq Bb$, \\ 

\emph{(ii)} & $B(a \wedge b)=Ba \wedge Bb$, \\ 

\emph{(iii)} & $B(a \wedge b) \leq a \cdot b$, \\ 

\emph{(iv)} & $B(a \cdot b)=B(a \wedge b)$, \\ 

\emph{(v)} & $B(a \cdot b)=Ba \cdot Bb$, \\ 

\emph{(vi)} & $Ba \vee Bb\leq B(a \vee b)$, \\ 

\emph{(vii)} & $B(a \to b) \leq Ba \to Bb$, \\ 

\emph{(viii)} & $B0=0$, \\ 

\emph{(ix)} & $B\neg a \leq \neg Ba$. \\

\end{tabular}
\end{lem}

\begin{proof2}
(i) Suppose $a \leq b$. Using {\bf{(BI)}}, it is enough to have $Ba \leq b$ and $Ba \vee \neg Ba = 1$. 
Now, the former follows by {\bf{(BE1)}} and the hypothesis, and the latter is {\bf{(BE2)}}. 

\smallskip

(ii) $B(a \wedge b) \leq Ba \wedge Bb$ follows from $a \wedge b \leq a, b$ using $B$-monotonicity. 
The other inequality follows using {\bf{(BI)}}, {\bf{(BE1)}}, and (iii) in Proposition \ref{infmon2}.

\smallskip

(iii) By (i) in Proposition \ref{infmon} and part (ii) we have $B(a \wedge b) = Ba \wedge Bb = Ba \cdot Bb$. 
The goal follows using $Ba \leq a$, $Bb \leq b$, and monotonicity of $\cdot$.   

\smallskip

(iv) From $a \cdot b \leq a \wedge b$ by (i), we get $B(a \cdot b) \leq B(a \wedge b)$. 
For the other inequality, using {\bf{(BI)}}, it is enough to have $B(a \wedge b) \leq a \cdot b$ and $B(a \wedge b)$ Boolean. 
Now, the former is (iii) and the latter follows from {\bf{(BE2)}}.

\smallskip

(v) As $Ba$ is Boolean, by (i) of Proposition \ref{infmon}, we have $Ba \wedge Bb=Ba \cdot Bb$. 
Moreover, by (ii), $B(a \wedge b)=Ba \wedge Bb$. We get our goal using (iv).

\smallskip

(vi) It follows using (i) ($B$-Monotonicity).

\smallskip

(vii) As $a \to b \leq a \to b$, we have $(a \to b) \cdot a \leq b$. 
Then, by $B$-monotonicity,  $B((a \to b) \cdot a) \leq Bb$. 
So, using (v), $B(a \to b) \cdot Ba \leq Bb$. 
So, $B(a \to b) \leq Ba \to Bb$.

\smallskip

(viii) It follows because $0$ is Boolean.

\smallskip

(ix) It follows from (vii), (viii), and $\neg a = a \to 0$.
\end{proof2}

Regarding the inequalities in the previous lemma, that is, $(iii), (vi), (vii)$, and $(ix)$, their reciprocals do not hold. 
Indeed, inequality $a \cdot b \leq B(a \wedge b)$ fails in the three-element G\"{o}del algebra ${\bf G}_3$ taking the top and the middle element. 
Inequality $B(a \vee b) \leq Ba \vee Bb$ fails in the algebra of Example \ref{5} taking $a$ and $b$ to be the coatoms $s$ and $t$. 
Also, inequality $\neg Ba \leq B\neg a$ fails in ${\bf G}_3$, taking $a$ to be the middle element. 
So, also inequality $Ba \to Bb \leq B(a \to b)$ fails.  

\medskip

Though $B$ may not exist for every element in a residuated lattice,  
$B$ exists in every finite residuated lattice. 

\begin{pro} \label{finB}
Let $\bf{A} \in \mathbb{RL}$ be finite. Then, $B$ exists in $A$. 
\end{pro}

\begin{proof2}
In a finite residuated lattice {\bf{A}}, for any $a \in A$, 
we have $Ba = \bigvee \{ b \in A: b \leq a$ and $b \vee \neg b = 1 \}$. 
It is enough to see that if $b_{1} \leq a$, $b_{1} \vee \neg b_{1} = 1$, $b_{2} \leq a$, and $b_{2} \vee \neg b_{2} = 1$, 
then (i) $b_{1} \vee b_{2} \leq a$ and (ii) $(b_{1} \vee b_{2}) \vee \neg (b_{1} \vee b_{2}) = 1$. 
Now, (i) follows immediately and (ii) follows using (ii) in Proposition \ref{infmon2}. 
\end{proof2}

On the other hand, there are infinite residuated lattices where $B$ does not exist. 
Indeed, we have the following example due to Franco Montagna (see \cite{AEM}).

\begin{pro}\label{F}
There is an (infinite) G\"{o}del algebra {\bf{A}} and $a \in A$ such that $Ba$ does not exist, i.e. where $B$ does not exist. 
\end{pro}

\begin{proof2}
Let $[0, \frac{1}{2}, 1]_{G}$ be the three-element G\"{o}del algebra. 
Let us consider

\begin{itemize}
\item[]$A_{1} = \{ a \in ([0, \frac{1}{2}, 1]_{G})^{\mathbb{N}}$ such that $\{ i \in {\mathbb{N}}: a_{i}=0 \}$ is finite$\}$,
\item[]$A_{2} = \{ a \in ([0, \frac{1}{2}, 1]_{G})^{\mathbb{N}}$ such that $\{ i \in {\mathbb{N}}: a_{i} \neq 0 \}$ is finite$\}$, and  
\item[]$A = A_{1} \cup A_{2}$.
\end{itemize}

\noindent The set $A$ is the domain of a subalgebra of $([0, \frac{1}{2}, 1]_{G})^{\mathbb{N} }$. 
Indeed, if $a,b \in A_1$, then $a\wedge b \in A_1$ and $a\to b \in A_1$, 
if $a,b \in A_2$, then $a \wedge b \in A_2$ and $a \to b \in A_1$, 
if $a \in A_1$ and $b \in A_2$, then $a \wedge b\in A_2$ and $a \to b \in A_2$, and 
if $a \in A_2$ and $b \in A_1$, then $a \wedge b \in A_2$ and $a \to b \in A_1$.
Also, $0 \in A_2$.
So, $A$ is the domain of a subalgebra $\bf{A}$ of $([0, \frac{1}{2}, 1]_{G})^{\mathbb{N}}$.

Now, take $a$ to be such that $a_{i}=1$ if $i$ is even and $a_{i}=\frac{1}{2}$ if $i$ is odd.
Next, consider the set $\{b \in A: b \leq a$ and $b$ is Boolean$\}$. 
It consists of all elements $b$ such that $b_{i}=0$ for all odd $i$ and for all but finitely many even $i$, 
and $b_{i}=1$ otherwise. 
It can be seen that this set has no maximum in $A$.
\end{proof2}

Actually, Montagna's example of Proposition \ref{F} can be generalized as follows. 

\begin{pro}\label{Montana-ext}

Let $\mathbb{V}$ be a variety of {\rm MTL}-algebras such that 
there is a linearly ordered algebra ${\bf A }\in \mathbb{V}$ with a proper filter $F$ 
(i.e. $\{ 1\} \subsetneq F \subsetneq A$), that is, such that $\bf A$ is not simple. 
Then, $\mathbb{V}$ contains an infinite algebra where $B$ does not exist. 

\end{pro}

\begin{proof2}
Let ${\bf D }\in \mathbb{V}$ be a chain and $F$ be a filter of $\bf A$ satisfying the hypothesis of the proposition. 
Let us define $F^\neg = \{ x \in D \mid \exists y \in F, x \leq \neg y \}$ and let $C = F \cup F^\neg$. 
It is easy to check that $C$ is the domain of a subalgebra of $\bf D$. 
Finally define the following sets: 

\begin{itemize}
\item[]$A_{1} = \{ a \in C^{\mathbb{N}}$ such that $\{ i \in {\mathbb{N}}: a_{i} \in F \}$ is finite$\}$,
\item[]$A_{2} = \{ a \in C^{\mathbb{N}}$ such that $\{ i \in {\mathbb{N}}: a_{i} \in F^\neg \}$ is finite$\}$, 
\item[]$A = A_{1} \cup A_{2}$.
\end{itemize}
One can check that again $A$ is the domain of  a subalgebra of ${\bf C^N}$, 
taking into account that if $x \in F$ and $y \in F^\neg$, then $x \land y, x * y, x \to y \in F^\neg$, 
and if  $x, y \in F^\neg$, then  $x \to y \in F$. 

Thus, $A$ is a subalgebra and taking an element $a$ such that $a_i = 1$ if $i$ is even and $a_i = b$, for a  given $b \in F \setminus \{1\}$, 
then the same argument as in Montagna's example proves that $Ba$ does not exist. 
\end{proof2}

For readers familiar with the main systems of mathematical fuzzy logic and their algebraic semantics (see \cite{Handbook}), 
we provide the following corollary with further examples of 
subvarieties of residuated lattices containing algebras where $B$ does not exist.

\begin{cor}
In  the following varieties of {\rm MTL}-algebras, there is an infinite algebra where $B$ does not exist:
\begin{itemize}
\item  the variety generated by any continuous t-norm,
\item the varieties generated by either the {\rm NM} t-norm or a {\rm WNM} t-norm. 
\footnote{Actually, this could be generalized in the following sense. 
In \cite{No07} Noguera proves that the variety generated by simple $n$-contractive MTL-chains is the variety of $S_n$-MTL algebras, 
i.e. MTL-algebras satisfying the law $x \vee \neg x^{n-1}= 1$. 
Therefore, any variety of $n$-contractive MTL-algebras that are not $S_n$-MTL has a chain with a proper filter. 
In particular, this is the case for the varieties of WNM and NM-algebras, since they are 3-contractives and are not $S_3$-MTL.}
\end{itemize}
\end{cor}

\begin{proof2}
In all these varieties there is an algebra {\bf A} satisfying the conditions of Proposition \ref{Montana-ext}. 

If the t-norm is either a G\"odel, Product\blue{,} or a {\rm WNM} t-norm (including {\rm NM}), 
then take as ${\bf A }$ the standard chain and as $F$ the positive elements respect to $\neg$, i.e., 
the elements such that $\neg x \leq x$. 
If the t-norm is {\L}ukasiewicz, then take  ${\bf A }$ as the Chang algebra and $F$ as the set of its positive elements. 
Finally\blue{,} if the continuous t-norm is a proper ordinal sum, 
then take ${\bf A }$ as the standard chain and $F = [a,1]$, 
where $a \in(0,1)$ is the end point of a component. 
It is clear that in all cases $F$ is a proper filter and thus Proposition \ref{Montana-ext} applies. 
 
\end{proof2}

Some papers (e.g. \cite {CC}) consider the notion of {\em compatible operation}. 
Operation $B$ is not compatible, that is, the congruences of $\mathbb{RL}$ and $\mathbb{RL}^B$ are not the same. 
To see this, take the three-element Heyting or G\"odel algebra ${\bf G}_3$ with universe $\{0, \frac{1}{2}, 1\}$ 
and the equivalence relation given by $\theta=\{(0,0), (\frac{1}{2},\frac{1}{2}), (\frac{1}{2},1), (1, \frac{1}{2}), (1,1) \}$. 
It holds that $\theta$ is a $\mathbb{RL}$-congruence, but not an $\mathbb{RL}^B$-congruence, as $B\frac{1}{2} = 0$ and $B1 = 1$.

Taking \emph{modality} to mean a finite combination of the unary operators $\neg$ and $B$, 
the next statement shows how many different modalities there are in $\mathbb{RL}^B$. 

\begin{pro}
In $\mathbb{RL}^B$ there are nine different modalities. 
They may be ordered as follows: on the one hand, the positive modalities $B \leq Id \leq \neg \neg \leq \neg B \neg$ 
(with also $B \leq B \neg \neg \leq \neg \neg$) and, on the other hand, the negative modalities  
$B \neg \leq \neg \leq \neg B \neg \neg \leq \neg B$. See Figure \ref{pnm}.
\end{pro}

\begin{proof2}
The inequalities are immediate. 
The reverse inequalities can be seen not to be the case by considering either the only atom in the three-element G\"{o}del algebra ${\bf G}_3$ 
or any of the {two non-comparable elements} of the Heyting algebra obtained by adding a top element to the Boolean algebra of 4 elements.
There are no other modalities, because if we apply operations $\neg$ and $B$ to the given nine modalities, 
we do not get anything new, as $\neg \neg B = B$, $BB = B$, and $B \neg B = \neg B$. 
\end{proof2}

\begin{figure} [ht]
\begin{center}

\begin{tikzpicture}

    \tikzstyle{every node}=[draw, circle, fill=white, minimum size=3pt, inner sep=0pt, label distance=1mm]
 
    \draw (0,0)		node (nBn)	[label=right:$\neg B \neg$] {} 
    
    -- ++(270:1cm)	node (nn)	[label=right:$\neg \neg$] {} 
    -- ++(225:1cm)	node (id)	[label=left:Id] {}
    -- ++(315:1cm)	node (B)	[label=below:$B$] {}
    -- ++(45:1cm)	node (Bnn)	[label=right:$B \neg \neg$] {} 
    -- ++ (nn); 
    
    \draw (5,0)	node (nB)	[label=right:$\neg B$] {} 
    
    -- ++(270:1cm)	node (nBnn)	[label=right:$\neg B \neg \neg$] {} 
    -- ++(270:1cm)	node (n)	[label=right:$\neg$] {}
    -- ++(270:1cm)	node (Bn)	[label=right:$B \neg$] {};
    
\end{tikzpicture}

\end{center}
\caption{\label{pnm} The positive and negative modalities of $\neg$ and $B$}
\end{figure}
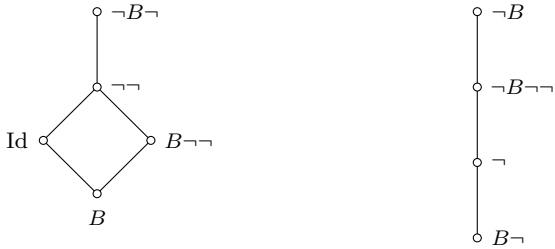

\subsection{An equational class}

It is natural to inquire whether the class $\mathbb{RL}^B$ is in fact an equational class. 
To this end, we start focusing our attention on the following equations, 
using $x \preccurlyeq y$ as an abbreviation for $x \vee y \approx y$: 

\vspace{5pt}

{\bf{(BI1)}} $Bx \preccurlyeq B(x \vee y)$, 

{\bf{(BI2)}} $B1 \approx 1$. 

\begin{lem}
Equations {\bf{(BI1)}} and {\bf{(BI2)}} hold in $\mathbb{RL}^B$. 
\end{lem}

\begin{proof2}
The given equations follow immediately from lemmas \ref{Blem}(ii) and \ref{LB}(vi), respectively.
\end{proof2}

We are also interested in the equation \medskip

{\bf{(BI3})} $B(x \vee \neg x) \preccurlyeq Bx \vee \neg x$, 
\medskip

\noindent but it is not easy to see that it holds in $\mathbb{RL}^B$. 
Towards this goal, we state and prove the following result.  

\begin{lem} \label{SE} 
In $\mathbb{RL}^B$ the following hold: 

\medskip

\emph{(i)} $(B(x \vee \neg x) \wedge x) \vee \neg(B(x \vee \neg x) \wedge x) \approx 1$, 

\smallskip

\emph{(ii)} $B(x \vee \neg x) \wedge x \preccurlyeq Bx$, 

\smallskip

\emph{(iii)} $B(x \vee \neg x) \wedge \neg Bx \preccurlyeq \neg x$. 
\end{lem}

\begin{proof2}
(i) Using Lemma \ref{GRL}(i), and using $T$ for the left hand side of the given equation, 
it is enough to get 

\medskip

(iv) $B(x \vee \neg x) \vee \neg B(x \vee \neg x) \approx 1$,

\smallskip

(v) $B(x \vee \neg x) \preccurlyeq T$, and

\smallskip

(vi) $\neg B(x \vee \neg x) \preccurlyeq T$.

\medskip

\noindent Part (iv) is immediate because of {\bf{(BE2)}}.

\smallskip

\noindent To see (v), using Proposition \ref{infmon}(viii), 
note that we have that $B(x \vee \neg x) \preccurlyeq x \vee \neg x$,  
(immediate using {\bf{(BE1)}}), $B(x \vee \neg x) \wedge x \preccurlyeq T$ (also immediate), 
and $B(x \vee \neg x) \wedge \neg x \preccurlyeq T$, 
which follows from $B(x \vee \neg x) \wedge \neg x \preccurlyeq \neg (B(x \vee \neg x) \wedge x)$, 
which holds because of Lemma \ref{GRL}(iv).

\smallskip

\noindent To see (vi), note that $B(x \vee \neg x) \wedge x \leq B(x \vee \neg x)$. 
So, using Lemma \ref{GRL}(iii), it follows that $\neg B(x \vee \neg x) \preccurlyeq \neg (B(x \vee \neg x) \wedge x)$. 
And so, $\neg B(x \vee \neg x) \preccurlyeq T$. 

\medskip

(ii) Use {\bf{(BI)}}, Part (i), and $B(x \vee \neg x) \wedge x \preccurlyeq x$. 

\medskip

(iii) Use Part (ii) and Proposition \ref{infmon}(x). 
\end{proof2}

\begin{pro} \label{BI3}
 The equation {\bf{(BI3)}} holds in $\mathbb{RL}^B$. 
\end{pro}

\begin{proof2}
Using Proposition \ref{infmon}(ix), it is enough to check the following three conditions: 

\medskip

(i) $Bx \vee \neg Bx \approx 1$,

\smallskip

(ii) $B(x \vee \neg x) \wedge Bx \preccurlyeq Bx \vee \neg x$, and

\smallskip

(iii) $B(x \vee \neg x) \wedge \neg Bx \preccurlyeq Bx \vee \neg x$.

\medskip

\noindent Now, (i) is immediate due to {\bf{(BE2)}} and 
(ii) is also immediate as $B(x \vee \neg x) \wedge Bx \preccurlyeq Bx$.
Regarding (iii), it follows from Lemma \ref{SE}(iii). 
\end{proof2}

\begin{rem} \label{VE}
 Arguing as in the proof of Proposition \ref{BI3} and noting that $B(x \vee \neg x) \wedge \neg Bx \preccurlyeq B \neg x$ follows 
 using {\bf{(BI)}} from Lemma \ref{SE}(iii) and the fact that the term $B(x \vee \neg x) \wedge \neg Bx$ is Boolean, 
 it may be seen that also the inequality $B(x \vee \neg x) \preccurlyeq Bx \vee B \neg x$ holds in $\mathbb{RL}^B$. 
\end{rem}

\begin{lem} \label{BME} 
$B$ is monotone just using equations.
\end{lem}
 
\begin{proof2}
 Suppose $x \vee y \approx y$. 
 Then, (i) $B(x \vee y) \approx By$. 
 Now, using {\bf{(BI1)}}, we have (ii) $Bx \vee B(x \vee y) \approx B(x \vee y)$. 
 From (i) and (ii) we get $Bx \vee By \approx By$. 
\end{proof2}

The following theorem answers positively the question posed above. 

\begin{teo} \label{ec}
$\mathbb{RL}^B$ is an equational class. 
An equational basis relative to $\mathbb{RL}$ is the following set of equations: 

\medskip

 {\bf{(BE1)}} $Bx \preccurlyeq x$, 
 
 {\bf{(BE2)}} $Bx \vee \neg Bx \approx 1$, 
 
 {\bf{(BI1)}} $Bx \preccurlyeq B(x \vee y)$, 
 
 {\bf{(BI2)}} $B1 \approx 1$, 
 
 {\bf{(BI3)}} $B(x \vee \neg x) \preccurlyeq Bx \vee \neg x$. 
\end{teo}

\begin{proof2}
It is enough to prove {\bf{(BI)}} using the given equations. 
Suppose (i) $b \leq a$ and (ii) $b \vee \neg b = 1$. 
From (i), using Lemma \ref{BME}, it follows (iii) $Bb \leq Ba$. 
From (ii), using {\bf{(BI2)}}, it follows $B(b \vee \neg b) = B1 = 1$, which, using {\bf{(BI3)}}, implies that $Bb \vee \neg b = 1$, 
which, using (iii) gives $Ba \vee \neg b = 1$, which, using Proposition \ref{infmon}(xi), implies $b\leq Ba$. 
\end{proof2}

\begin{rem}
Note that, as expected, the just given proof of {\bf{(BI)}} only uses equations {\bf{(BI1)}}, {\bf{(BI2)}}, and {\bf{(BI3)}}.  
\end{rem}

It is also natural to inquire whether the given equations are independent. 

\begin{pro}
The set \{\emph{{\bf{(BE1)}}, {\bf{(BE2)}}, {\bf{(BI1)}}, {\bf{(BI2)}}, {\bf{(BI3)}}}\} is independent.
\end{pro}

\begin{proof2}
To see that {\bf{(BE1)}} is independent of the rest, take the three-element G\"odel algebra ${\bf G}_3$
and define $B0 = 0$, and $Ba = 1$, if $a$ is not $0$.

\smallskip

\noindent To see that {\bf{(BE2)}} is independent of the rest, take the four-element G\"odel chain ${\bf G}_4$ 
and define $B0 = 0$, $B1 = 1$, $Ba = 0$, for the only atom $a$ of the chain, and $Bc = c$, for the remaining element $c$.

\smallskip

\noindent To see that {\bf{(BI1)}} is independent of the rest, take the G\"odel algebra ${\bf G}_3 \times {\bf G}_3$. 
If $a$ is any of the four Boolean elements, then put $Ba = a$, else put $Ba = 0$.

\smallskip

\noindent To see that {\bf{(BI2)}} is independent of the rest, take again ${\bf G}_3$, but now define $Ba = 0$, for every $a$.

\smallskip

\noindent Finally, to see that {\bf{(BI3)}} is independent of the rest, 
take the four-element Boolean ${\bf G}_2 \times {\bf G}_2$ algebra and define, 
for any $a$, if $a = 1$, then $Ba = 1$, else $Ba = 0$.  
\end{proof2}

\subsection{Subdirectly irreducible $\mathbb{RL}^B$-algebras}
In this section we show the subdirectly irreducible members of $\mathbb{RL}^B$ 
are those whose Boolean elements are only the top and bottom elements. 

\begin{defi}
Let {\bf{A}} $\in \mathbb{RL}^B$. 
A set $F$ contained in $A$ is said to be a $\mathbb{RL}^B$-filter iff for all $a, b \in A$ it satisfies

\emph{(1)} $1 \in F$, 

\emph{(2)} if $a \in F$ and $a \leq b$, then $b \in F$, 

\emph{(3)} if $a, b \in F$, then $a \cdot b \in F$, 

\emph{(4)} if $a \in F$, then $Ba \in F$. 
\end{defi}

\begin{pro}
Let {\bf{A}} $\in \mathbb{RL}^B$. 
The lattice of $\mathbb{RL}^B$-congruences is isomorphic to the set of $\mathbb{RL}^B$-filters. 
Indeed, let  $f: Con({\bf{A}}) \longrightarrow Fil({\bf{A}})$ be defined by:
if $\equiv$ is a $\mathbb{RL}^B$-congruence, then $f(\equiv)$ is the $\mathbb{RL}^B$-filter $F_\equiv = \{a \in A: a \equiv 1\}$. 
Then, the function $f$ is an isomorphism such that if $F$ is a $\mathbb{RL}^B$-filter, 
then $f^{-1}(F)$ is the $\mathbb{RL}^B$-congruence $\equiv_F$ defined by $a \equiv_F b$ iff $a \to b, b \to a \in F$.
\end{pro}

\begin{proof2}
It is obvious that  $F_\equiv $ is a $\mathbb{RL}^B$-filter. 
In order to prove that $\equiv_F $ is a congruence we need to prove that if $a \equiv_F b$, then $Ba \equiv_F Bb$, 
since the other conditions are known to be true for any residuated lattice. 
So, suppose $a \to b$ and $b \to a \in F$. 
Then, by the fourth condition in the definition of filter, $B(a \to b) \in F$.
Now, using  (vi) in Lemma \ref{LB} and the second condition in the definition of filter, it follows that $Ba \to Bb \in F$. 
Analogously, we obtain that $Bb \to Ba \in F$. 
Finally, it is also obvious that $f^{-1} \circ f = Id$.
\end{proof2}

Now we can characterize a family of $\mathbb{RL}^B$-filters.

\begin{pro}
Let {\bf{A}} $\in \mathbb{RL}^B$. 
If $a \in B(A)$, then $F_a = [a,1] = \{x \in A: a \leq x \leq 1\}$ is a $\mathbb{RL}^B$-filter.
\end{pro}

\begin{proof2}
It is obvious that $F_a$ satisfies the first two conditions of a $\mathbb{RL}^B$-filter. 
The third is an easy consequence of the fact that if $a \in B(A)$, then $a \ast x = a \land x$ and thus 
if $x,y \in F_a$, then $a=a \land y \leq x \ast y$ and thus $x \ast y \in F_a$. 
Finally, if $x \in F$, then $a=Ba \leq Bx$.
\end{proof2}

From now on, $F_a$ denotes the principal filter defined by $a\in B(A)$.

In order to characterize the subdirectly irreducible $\mathbb{RL}^B$-algebras, we will use the result of  \cite[Theorem 97]{Sz}: 
an algebra {\bf{A}} is subdirectly reducible iff 
there exists a family of non-trivial congruences $\sigma_i$ such that their intersection is the identity. 
In our case, this means that {\bf{A}} is subdirectly irreducible iff 
there is a unique coatom in the lattice of $\mathbb{RL}^B$-congruences of {\bf{A}}.

\begin{pro} \label{si}
Let {\bf{A}} $\in \mathbb{RL}^B$. 
Then, {\bf{A}} is subdirectly irreducible iff $B(A) = \{0,1\}$.
\end{pro}

\begin{proof2}
Observe first that if $F$ is a $\mathbb{RL}^B$-filter of $A$, then $F$ contains a Boolean element $a$ 
(by the third condition of $\mathbb{RL}^B$-filter) and, thus, $F$ contains $F_a$. 
So, to obtain the intersection  of the non-trivial $\mathbb{RL}^B$-filters of {\bf{A}} 
it is enough to compute the intersection of the filters $F_a$. 
However, this intersection is not the identity iff there exists a unique Boolean element $a$ such that $a$ is a coatom of $B(A)$. 
So, being $B(A)$ a Boolean algebra, this implies that $B(A) = \{0,1\}$.
\end{proof2}

\section{Comparing $B$ with other operations}

Operation $B$ is strongly related to other operations considered in the literature, e.g.  
the Monteiro-Baaz $\Delta$ and an operation defined with the join-complement $D$. 
In this section we study these relationships. 

\subsection{Comparing $B$ with $\Delta$}

The operation $\Delta$ was already considered by Monteiro in his paper on symmetric Heyting algebras in 1980 (see \cite{Mon}). 
Monteiro considered the same definitions of possibility and necessity operations given by  Moisil in \cite{Moi1} (see p. 67 in \cite{Mon}). 
However, instead of using Moisil's notation, Monteiro used $\nabla$ and $\Delta$, respectively. 
When so doing, he did not explicitly mention Moisil. 
However, many works by Moisil appear in the list of references of \cite{Mon}, including \cite{Moi1} and \cite{Moi2}. 
Monteiro also considered the $\Delta$ operator in the setting of linear symmetric Heyting algebras and 
studied the properties of $\Delta$ in the totally linear case (see \cite[Ch. 5, Sect. 3]{Mon}). 
In 1996, independently, Baaz in \cite{Ba} considered an expansion of G\"odel logic with a connective he also called $\Delta$ 
satisfying certain axioms and the rule $\varphi/\Delta \varphi$. 
Although he did not cite Monteiro, 
the proposed axioms are equivalent to the properties that Monteiro proved for his $\Delta$ operator 
in the framework of totally linear symmetric Heyting algebras. 
Baaz also provided a deduction theorem using $\Delta$: $\Gamma, \varphi \vdash \psi$ iff $\Gamma \vdash\Delta \varphi \to \psi$.
In 1998, H\'ajek considered Baaz's $\Delta$ in the context of BL-algebras and BL logic (see pp. 57-61 \cite{Ha}). 
He gave for $\Delta$ exactly the same axioms as Baaz presented in \cite{Ba} for G\"odel logic. 
He observed that all $\Delta$ axioms make it behave like a necessity operator, 
with the exception of the axiom $\Delta (\varphi \vee \psi) \to (\Delta \varphi \vee \Delta \psi)$, 
that is characteristic of possibility operations (see Remark 2.4.7 of \cite{Ha}). 
The $\Delta$ operation has also been studied in the more general context of Mathematical fuzzy logic, 
see several chapters in the handbook \cite{Handbook}. 
More recently, in \cite{AEM} the authors study, among other things, 
the expansion of FL$_{ew}$ with the $\Delta$ operator and show that it is conservative.

In MTL, $\Delta$ can always be defined over chains, namely as $\Delta 1 = 1$ and $\Delta x = 0$
for all $x \neq 0$, and thus, $\Delta$ and $B$ over MTL-chains coincide. 
But there are (non-linearly) MTL-algebras where $\Delta$ does not exist. 
Nevertheless, this is not a problem because MTL is semilinear, and the semantics of $\Delta$ over chains is clear. 
However, there is not a  clear semantical interpretation of the axioms of $\Delta$ in the general context of residuated lattices. 

In the context of a residuated lattice, 
the operator $\Delta$ is introduced e.g. in \cite{AEM} by the same equations as in MTL or BL (cf.\cite[p. 58]{Ha}): 

\medskip

\begin{tabular}{ll}
 {\bf{($\Delta$E1)}} $\Delta x \preccurlyeq x$, \\ 
 {\bf{($\Delta$E2)}} $\Delta x \vee \neg \Delta x \approx 1$, \\ 
 {\bf{($\Delta$I1)}} $\Delta(x \vee y) \preccurlyeq \Delta x \vee \Delta y$, \\ 
 {\bf{($\Delta$I2)}} $\Delta 1 \approx 1$, \\ 
 {\bf{($\Delta$I3)}} $\Delta x \preccurlyeq \Delta \Delta x$, \\ 
 {\bf{($\Delta$I4)}} $\Delta (x \to y) \preccurlyeq \Delta x \to \Delta y$, 
\end{tabular}

\medskip

\noindent where, again, $x \preccurlyeq y$ abbreviates $x \vee y \approx y$. 
Note that {\bf{($\Delta$I3)}} may be derived from the rest: 
it is enough to check that an operator satisfying the rest of the equations, 
satisfies all the equations in Theorem \ref{ec}, and hence  the quasi-equation {\bf{(BI)}} as well; 
then use (iii) of Lemma \ref{Blem}. 
Also, regarding their defining equations, 
the only difference between $\Delta$ and $B$ is that $\Delta$ satisfies $\Delta(x \vee y) \preccurlyeq \Delta x \vee \Delta y$, 
whereas $B$ only satisfies the particular case $y=\neg x$, that is, 
$B$ only satisfies $B(x \vee \neg x) \preccurlyeq Bx \vee B \neg x$, as stated in Remark  \ref{VE}. 

We will denote by $ \mathbb{RL}^\Delta$ the class of residuated lattices expanded with $\Delta$. 

It will be useful to bear in mind the following fact.

\begin{lem}\label{DB}
 Let ${\bf{A}} \in \mathbb{RL}^\Delta$ and $a \in A$. 
 Then, $\Delta a = a$ iff $a$ is Boolean. 
\end{lem}

\begin{proof}
 Supposing $\Delta a = a$, using {\bf{($\Delta$E2)}} it follows that $a$ is Boolean. 
 On the other hand, suppose $a$ is Boolean. 
 Considering {\bf{($\Delta$E1)}}, it is enough to prove that $a \leq \Delta a$. 
 By Lemma \ref{GRL}(ii), it is enough in turn to prove $\Delta a \vee \Delta \neg a = 1$, 
 $a \cdot \Delta a \leq \Delta a$, and $a \cdot \Delta \neg a \leq \Delta a$. 
 The first condition holds using {\bf{($\Delta$I1)}} and {\bf{($\Delta$I2)}}, since $a$ is Boolean. 
 The second condition is immediate. 
 For the third, observe that $a \cdot \Delta \neg a \leq a \cdot \neg a = 0 \leq \Delta a$.  
\end{proof}

Actually, $\Delta$ is somewhat stronger than $B$ in the following sense. 

\begin{pro} \label{DimB} Let $\bf A \in \mathbb{RL}$. 
 If $\Delta$ exists in $\bf A$, then so does $B$, with $B = \Delta$. 
\end{pro}

\begin{proof2}
Considering Theorem 1, all we have to see is that $\Delta$ satisfies the equational basis given for $B$. 
This is immediate excepting {\bf{(BI1)}}. 
Let us see that the equation $\Delta x \preccurlyeq \Delta (x \vee y)$ also holds. 
As we have $x \to (x \vee y) \approx 1$, using {\bf{($\Delta$I2)}} and {\bf{($\Delta$I4)}} we get 
$1 \preccurlyeq \Delta x \to \Delta (x \vee y)$, which gives $\Delta x \preccurlyeq \Delta (x \vee y)$. 
\end{proof2}

On the other hand, we have the following result.

\begin{pro}
There exist finite residuated lattices where $B$ exists, but $\Delta$ does not.
\end{pro}

\begin{proof2}
 Using Proposition \ref{finB}, it follows that $B$ exists in the G\"odel algebra of Example \ref{5}, as the algebra is finite. 
 Now, take its coatoms $s$ and $t$.  
 To see that $\Delta$ does not exist, note that {\bf{($\Delta$E1)}} and {\bf{($\Delta$E2)}} imply that $\Delta s = \Delta t = 0$. 
 So, $\Delta s \vee \Delta t = 0$. 
 However, $\Delta (s \vee t) = 1$, due to {\bf{($\Delta$I2)}}. 
 Then, {\bf{($\Delta$I1)}} is not satisfied. 
\end{proof2}

Example \ref{5} makes clear the basic difference between $\Delta$ and $B$, when we define them over MTL-algebras. 
It is well known that MTL$^\Delta$, the expansion of MTL with $\Delta$, is semilinear, i.e.  
each  algebra of the variety is a subdirect product of lineraly ordered MTL$^\Delta$-algebras. 
Moreover, we have seen that $\Delta$ and $B$ coincide over chains. 
Thus, Example \ref{5} proves that MTL$^B$, the expansion of MTL with $B$, is not semilinear. 
In fact, this was already clear from Proposition \ref{si}, since there exist subdirectly irreducible MTL$^B$-algebras 
(like the one defined in Example \ref{5}) that are not linearly ordered.

\subsection{Comparing $B$ with an operation using the join-complement} \label{join-complement}

The join-complement operation $D$ has a long history.
In 1919, Skolem considered lattices expanded with both meet and join relative complements (see \S 2 of \cite{Sk1} or pp. 77-85 of \cite{Sk2}).   
He just worked from an algebraic point of view. 
He noted that existence of both top $1$ and bottom $0$ is implied. 
Also, he briefly considered the meet and join-complements, for which, for an arbitrary argument $a$, 
he used the notations $\frac{0}{a}$ and $1-a$, respectively. 

In 1942, Moisil defined possibility as $\neg \neg$ and necessity as $DD$ in a logical setting where he had both 
intuitionistic negation $\neg$ and its dual $D$ (see \S 4 of \cite{Moi1} or p. 365 in \cite{Moi2}). 
He did not mention Skolem. 
In 1949, Ribenboim proved that distributive lattices with $D$ form an equational class (see \cite{Ri}).  
In fact, the meet is not needed, as the class with join and join-complement $D$ is already an equational class. 
In 1974, Rauszer, mainly considering algebraic aspects, studied a logic with conjunction, disjunction, conditional, and its dual (see \cite{Rau}). 
She also included both intuitionistic negation $\neg$ and its dual $D$, though these can be easily defined.  
Her axiomatization included the expected axioms plus the rules \emph{modus ponens} and $\varphi/\neg D\varphi$. 
She also provided a deduction theorem using $(\neg D)^n$.
She neither mentioned Skolem nor Moisil. 

In the context of a join semi-lattice $\bf A$, 
it is possible to postulate the existence of the join-complement 
$Da = \min\{b \in A:$ for all $c \in A$, $c \leq a \vee b \},$ for $a \in A$. 
This is equivalent to the following two conditions: 

\medskip

\begin{tabular}{ll}
 {\bf{(DI)}} $b \leq a \vee Da$, for all $a, b \in A$, \\
 {\bf{(DE)}} for any $a, b \in A$, if for all $c \in A$, $c \leq a \vee b$, \\
 then $Da \leq b$.  
\end{tabular}

\medskip

\noindent In a join semi-lattice the existence of $D$ implies the existence of both top $1 = a \vee Da$, for any $a$, 
and bottom $0 = D(a \vee Da)$, for any $a$.  
Moreover, $D$ can be equationally characterized by the following three equations, 
where, again, we use $x \preccurlyeq y$ as an abbreviation for $x \vee y \approx y$: 

\medskip

\begin{tabular}{ll}
{\bf{(DI)}}  & $y \preccurlyeq x \vee Dx$, \\ 
{\bf{(DE1)}} & $D(x \vee Dx) \preccurlyeq y$, \\ 
{\bf{(DE2)}} & $Dy \preccurlyeq x \vee D(x \vee y)$.  
\end{tabular}

\medskip

In what follows, $\mathbb{RL}^D$ will denote the class of residuated lattices expanded with an operation $D$ satisfying these equations. 
Obviously, by definition, $\mathbb{RL}^D$ is an equational class. 
Notice that in a residuated lattice $\bf A$, having in the signature the symbols $0$ and $1$ for the bottom and top elements, respectively,
the above definition of $D$ can be simplified to $Da = \min\{b \in A: a \vee b = 1\}$, 
and the condition {\bf{(DE)}} simplifies to be 

\medskip

\begin{tabular}{ll}
{\bf{(DE$^\prime$)}}  & for any $a, b \in A$, if $a \vee b = 1$, then $Da \leq b$. 
\end{tabular}

\medskip

\noindent Moreover the equations {\bf{(DI)}} and {\bf{(DE1)}} can also be simplified to: 

\medskip

\begin{tabular}{ll}
{\bf{(DI$^\prime$)}}  & $x \vee Dx \approx 1$, \\ 
{\bf{(DE1$^\prime$)}} & $D1 \approx 0$. 
\end{tabular}

\begin{rem} Note that, while $x \preccurlyeq \neg\neg x$ holds in $\mathbb{RL}$, 
from {\bf{(DE$^\prime$)}} and {\bf{(DI$^\prime$)}} it follows that $ DDx \preccurlyeq x$ holds in $\mathbb{RL}^D$. 
Note also that in a Heyting algebra $D$ 
is the dual of $\neg$, 
since in that case $\neg$ coincides with the meet complement. 
\end{rem}

As in the case of $B$, $D$ may not exist in some residuated lattices, but it always exists in the finite ones. 

\begin{pro} 
 Let {\bf{A}} be a finite residuated lattice. 
 Then $D$ exists in $A$.
\end{pro}

\begin{proof2}
It is enough to prove that $\bigwedge\{ b \in A: a \vee b = 1 \}$ exists in $A$. 
For that, it is enough to see that if $a \vee b_{1} = 1$ and $a \vee b_{2} = 1$, then $a \vee (b_{1} \wedge b_{2}) = 1$. 
Now, from the antecedent it follows that $(a \vee b_{1}) \cdot (a \vee b_{2}) = 1$ and 
using twice the distributive law of $\cdot$ with respect to $\vee$, 
we have that $(a \cdot a) \vee (a \cdot b_{2}) \vee (b_{1} \cdot a) \vee (b_{1} \cdot b_{2}) = 1$. 
Any subterm $t$ of the left-hand term is such that $t \leq a \vee (b_{1} \wedge b_{2})$.
\end{proof2}

Following \cite{Rau} and \cite{RZ}, we consider now the compound operation $\neg D$ and its relation to $B$. 
First, let us state the following fact. 

\begin{lem} \label{NDM}
 Let $\bf{A}\in \mathbb{RL}^D$ and $a, b \in A$. Then,  
 
 \emph{(i)} $\neg Da \leq a$,  
 
 \emph{(ii)} if $a \leq b$, then $Db \leq Da$ and $\neg Da \leq \neg Db$.
\end{lem}

\begin{proof2} (i) follows from $a \vee Da = 1$ using Lemma \ref{GRL}(vi). 

(ii) Assume $a \leq b$. Then, $1 = a \lor Da \leq b \lor Da$. 
Hence, by {\bf{(DE)}} we have $Db \leq Da$. 
Now, apply Lemma \ref{GRL}(iii) to get $\neg Da \leq \neg Db$. 
\end{proof2}

In \cite[Section 5]{CEB} the authors prove a result about iterations of the operation $\neg D$
in the context of meet-complemented distributive lattices with $D$. 
Once trivially adapted to  $\mathbb{RL}^D$, it is the following fact.  

\begin{pro} 
 \emph{(i)} For any natural $n > 0$, let $\mathbb{RL}^{D_n}$ be the subvariety of $\mathbb{RL}^D$ defined 
 by adding to those of $\mathbb{RL}^D$ the following equation:
 $$(\neg D)^{n+1} x  \approx (\neg D)^n x.$$ 
 Then, the sequence of varieties 
 $\mathbb{RL}^{D_1} \subset \mathbb{RL}^{D_2} \subset \ldots \subset \mathbb{RL}^{D_n} \subset \ldots $ 
 is strictly increasing. 
 
 \emph{(ii)} There are algebras of $\mathbb{RL}^D$ where none of the equations given in \emph{(i)} hold. 
\end{pro}

Next, consider the following example of a Heyting algebra where $B$ exists but $D$ does not. 

\begin{example} \label{ByesDno}
 $B$ exists in the Heyting algebra $1+(\mathbb{N}\times \mathbb{N})^{\partial}$ of Figure \ref{1+(NxN)op}, 
 where $(\mathbb{N}\times \mathbb{N})^{\partial}$ is obtained `turning upside down' the partial order $\mathbb{N}\times \mathbb{N}$. 
 In that Heyting algebra, $Ba=1$ if $a=1$ else $Ba=0$. 
 However, $D$ does not exist for the elements $(0,n)$ and $(n,0)$, with non-zero $n$. 
\end{example}

\begin{figure} [ht]
\begin{center}

\begin{tikzpicture}

  \tikzstyle{every node}=[draw, circle, fill=white, minimum size=3pt, inner sep=0pt, label distance=1mm]
 
    \draw (0,0)
    
           ++(45:.5cm)	node (02) [label=left:(0{,}2)] {}
        -- ++(45:1)	node (01) [label=left:(0{,}1)] {}
        -- ++(45:1)	node (00) [] {}
        -- ++(315:1)	node (10) [label=right:(1{,}0)] {}
        -- ++(315:1)	node (20) [label=right:(2{,}0)] {}
        -- ++(225:1)	node (21) [] {}
        -- ++(135:1)	node (11) [] {}        
        -- ++(225:1)	node (12) [] {};
       
    \draw (02)--(12);
    \draw (01)--(11); 
    \draw (10)--(11);
       
    \draw [dotted, very thick] (02)--(0,0); 
    \draw [dotted, very thick] (12)--(315:1);
    \draw [dotted, very thick] (12)--(333:1.5);
    \draw [dotted, very thick] (21)--(342:2.25); 
    \draw [dotted, very thick] (21)--(345:2.95);
    \draw [dotted, very thick] (20)--(359:3.6);
    
    \draw (315:2.55) node [] {};
    
\end{tikzpicture}

\end{center}
\caption{\label{1+(NxN)op} The residuated lattice $1+(\mathbb{N}\times \mathbb{N})^{\partial}$, 
where $B$ exists but $D$ does not}
\end{figure}
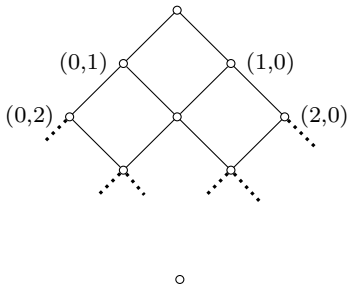

There are also $\mathbb{RL}$-algebras where $D$ exists, 
but $B$ does not (see the end of Section 2 of \cite{CEB} for an example of a Heyting algebra).  
Note that in Franco Montagna's example, neither $B$ nor $D$ exist. 

In the following case, existence of $D$ implies existence of $B$.

\begin{pro} \label{inStoneDgivesB}
Let $\bf{A} \in \mathbb{RL}$ and $\neg a \vee \neg \neg a = 1$, for all $a \in A$. 
Then, if $D$ exists in {\bf{A}}, then $B$ also exists in {\bf{A}}, with $B = \neg D$.
\end{pro}

\begin{proof2}
Let $a \in A$. 
Then, $\neg Da$ exists in $A$. 
We have to see (i) $\neg Da\leq a$, (ii) $\neg Da \vee \neg \neg Da = 1$, and 
(iii) if $b \leq a$ and $b \vee \neg b = 1$, then $b \leq \neg Da$. 
Now, (i) holds as seen in Lemma \ref{NDM}(i) and 
(ii) follows from the hypothesis that $\neg a \vee \neg \neg a = 1$, for any $a \in A$. 
To see (iii), suppose (iv) $b \leq a$ and (v) $b \vee \neg b = 1$. 
Due to Lemma \ref{NDM}(ii), we have that (iv) implies $Da \leq Db$ and, using {\bf{(DE)}}, (v) implies $Db \leq \neg b$. 
So, by $\leq$-transitivity it follows that $Da \leq \neg b$. 
Then, in a residuated lattice we have $b \leq \neg Da$. 
\end{proof2}

\begin{rem}
Given the conditions of Proposition \ref{inStoneDgivesB}, 
taking any of the coatoms in the algebra of Example \ref{5}, 
it is easy to see that $Dx \approx \neg Bx$ does not hold. 
Also, the reciprocal of Proposition \ref{inStoneDgivesB} is not the case, 
as the algebra in Example \ref{ByesDno} satisfies the equation $\neg x \vee \neg \neg x \approx 1$ and 
$B$ exists in that algebra, but $D$ does not exist. 
\end{rem}

Taking into account De Morgan laws valid in any MTL-algebra 
\footnote{Notice that not all instances of De Morgan laws are valid in the variety of MTL-algebras, 
for instance the equations  $\neg(x \land y) \approx \neg x \lor \neg y$ and $\neg(x \lor y) \approx \neg x \land \neg y$  are valid, 
but $x \land y \approx \neg(\neg x \lor \neg y)$ is not.}, 
we can easily obtain the following consequence of the previous proposition.

\begin{cor}
Let {\bf{A}} be a SMTL-algebra, 
i.e. an MTL-algebra such that for all $a \in A$, $a \wedge \neg a = 0$. 
Then, for any $a \in A$, if $Da$ exists, then so does $Ba$, and $Ba = \neg Da$.
\end{cor}

\begin{lem} \label{L-D} 
 Let $\bf{A} \in \mathbb{RL}^D$ and $a \in A$. 
 Then, the following are equivalent:

 \vspace{5pt}
 
 \emph{(i)} $a$ is Boolean, 
 
 \emph{(ii)} $\neg Da = a$, 
 
 \emph{(iii)} $Da = \neg a$. 
\end{lem}

\begin{proof2}
 (i) $\Rightarrow$ (ii) Suppose $a \vee \neg a = 1$. 
 Then, using {\bf{(DE)}}, $Da \leq \neg a$. 
 Then, $a \leq \neg Da$. 
 Now, by Lemma \ref{NDM}(i) we have $\neg Da \leq a$. 
 So, $\neg Da = a$.
 
 (ii) $\Rightarrow$ (iii) Suppose $\neg Da = a$. 
 Then, $\neg \neg Da = \neg a$. 
 As $Da \leq \neg \neg Da$, we have that $Da \leq \neg a$. 
 Now, by {\bf{(DI)}}, $a \vee Da = 1$. 
 So, also $\neg a \leq Da$. 
 Then, $Da = \neg a$.
 
 (iii) $\Rightarrow$ (i) Suppose $Da = \neg a$. 
 As using {\bf{(DI)}} we have $a \vee Da = 1$, it follows that $a \vee \neg a = 1$. 
\end{proof2}

As a direct consequence we have the following fact. 

\begin{cor}
Let {\bf{A}} be a residuated lattice where both $B$ and $D$ exist. 
Then, $Da \leq \neg Ba$, for all $a \in A$. 
Equivalently, $Ba \leq \neg Da$, for all $a \in A$.
\end{cor}

\begin{proof2}
Let $a \in A$. We have that $Ba \leq a$. 
Hence, using Lemma \ref{NDM}(ii), $Da \leq DBa$. 
Now, since $Ba$ is Boolean, using Lemma \ref{L-D} it follows that $Da \leq \neg Ba$. 
\end{proof2}

\begin{rem}
The equality $B \approx \neg D$ does not hold. 
Indeed, consider the join-irreducible coatom $c$ in the Heyting algebra in Figure \ref{rn8}, where $0 = Bc < \neg Dc$. 
\end{rem}

\begin{figure} [ht]
\begin{center}

\begin{tikzpicture}

    \tikzstyle{every node}=[draw,circle,fill=white,minimum size=3pt, inner sep=0pt, label distance=1mm]

    \draw (0,0) node (1) [] {}
        -- ++(315:1cm) node (rc) [label=right:$c$] {}
        -- ++(225:1cm) node (d) [] {}
        -- ++(315:1cm) node (nD) [label=right:$\neg Dc$] {}
        -- ++(225:1cm) node (0) [label=right:$Bc$] {}
        -- ++(135:1cm) node (rb) [] {}
        -- ++(135:1cm) node (c) [label=left:$Dc$] {}
        -- ++(45:1cm) node (e) [] {}
        -- (1);

        \draw (d) -- (e); 
        \draw (d) -- (rb);

        \end{tikzpicture}
\end{center}
\caption{\label{rn8} Behaviour of $B$ and $D$ in a coatom of a residuated lattice}
\end{figure}
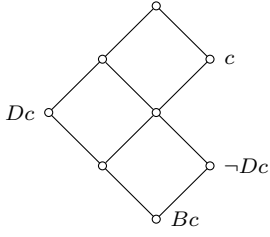

\begin{lem} \label{nmB} 
 Let {\bf{A}} be a $\mathbb{RL}^D$-algebra and $a, b \in A$. 
 We have that if $b \leq a$ and $b \vee \neg b = 1$, then $b \leq \neg Da$.
\end{lem}

\begin{proof2}
 Suppose $b \leq a$. 
 Then, $\neg Db \leq \neg Da$. 
 Now, using the hypothesis $b \vee \neg b = 1$ and Lemma \ref{L-D}, we have that $\neg Db = b$. 
 So, $b \leq \neg Da$.
\end{proof2}

\begin{pro}
Let {\bf{A}} be a $\mathbb{RL}^D$-algebra, let $a \in A$, 
and let us have a finite number of elements $b \in A$ such that $b \leq a$. 
Then, $Ba = (\neg D)^{n}a$, for some $n \in \mathbb{N}$. 
\end{pro}

\begin{proof2}
 In the case $a$ is Boolean, $Ba = (\neg D)^{0}a$. 
 In the case $a$ is not Boolean, take $\neg Da$. 
 Now, Lemma \ref{nmB} says we will not be missing Boolean elements. 
 Repeating the procedure we will find the first Boolean below $a$.
\end{proof2}

\begin{pro} 
 (i) In every algebra of $\mathbb{RL}^{D_n}$, for any natural number $n \geq 0$, 
 $B$ exists, with $B = (\neg D)^n$. 
 
 (ii) There are $\mathbb{RL}^D$-algebras where $B$ does not exist. 
\end{pro}

\begin{proof2}
 (i) It suffices to see that $(\neg Dx)^n$ satisfies {\bf{(BE1)}}, {\bf{(BE2)}}, and {\bf{(BI)}}. 
 It always satifies {\bf{(BE1)}}, as it is easily seen by induction using $\neg Da \leq a$ and $\leq$-transitivity. 
 We get {\bf{(BE2)}} applying Lemma \ref{L-D} on the hypothesis that $(\neg D)^{n+1}=(\neg D)^n$. 
 Finally, to get {\bf{(BI)}}, suppose both (i) $b \leq a$ and (ii) $b \vee \neg b =1$. 
 From (i) it follows (iii) $\neg Db \leq \neg Da$. 
 From (ii), using Lemma \ref{L-D}, we get (iv) $\neg Db = b$. 
 From (iii) and (iv) we get $b \leq \neg Da$. 
 Repeat the argument $n$ times to get $b \leq (\neg D)^na$. 
 
 (ii) cf. end of Section 2 in \cite{CEB}. 
\end{proof2}

In the next proposition we will use the following  De Morgan properties for  $\neg$ and $D$. 
In the proof we use the abreviation $\{x_i\}$ for $\{x_i \in A: i \in I\}$.

\begin{lem} \label{DM}
 Let {\bf{A}} be a complete $\mathbb{RL}^D$-algebra. 
 Then, both (i) $\neg \bigvee \{ x_{i} \} = \bigwedge \{ \neg x_{i} \}$ and 
 (ii) $D \bigwedge \{ x_{i} \} = \bigvee \{ D x_{i} \}$. 
\end{lem}

\begin{proof2} We prove only (ii), (i) is already known. 
By {\bf{(DI)}} we have that $ x_{j} \vee Dx_{j} = 1$. 
Then, $ x_{j} \vee \bigvee \{ Dx_{i}\} = 1$ for all $j \in I$ and so, $\bigwedge \{x_{i}\} \vee \bigvee \{ Dx_{i}\} = 1$. 
So, by {\bf{(DE)}} we obtain that $D\bigwedge \{ x_{i}\} \leq \bigvee \{D x_{i}\}$. 
  
For the other inequality, using {\bf{(DI$^\prime$)}}, we have that  $\bigwedge \{ x_{i} \} \vee D\bigwedge \{ x_{i} \} = 1$. 
Now, from $\bigwedge \{ x_{i} \} \leq x_{j}$ and $x_{j} \leq x_{j} \vee D \bigwedge \{x_{i}\}$ we obtain
$\bigwedge \{ x_{i} \} \leq x_{j} \vee D \bigwedge \{x_{i}\}$, 
and taking into account that $D\bigwedge\{ x_{i} \} \leq x_{j} \vee D\bigwedge\{ x_{i} \}$, we have that  
$\bigwedge \{ x_{i} \} \vee D \bigwedge\{ x_{i} \} \leq x_{j} \vee D (\bigwedge \{x_{i}\}$. 
Now, by {\bf{(DI$^\prime$)}} we obtain that $\{x_i\} \vee D \bigwedge\{x_{i}\} = 1$, which, by {\bf{(DE)}}, 
implies $D x_j \leq D \bigwedge \{x_{i}\}$, for all $j \in I$. 
Thus, we finally get $\bigvee \{ Dx_{i} \} \leq D\bigwedge \{ x_{i} \}$. 
\end{proof2}

In the next proposition, $\mathbb{N}$ and $\mathbb{N}^{+}$ denote the set of natural numbers including $0$ and excluding $0$, respectively. 

\begin{pro}
 Let {\bf{A}} be a complete $\mathbb{RL}^D$-algebra. 
 
 \noindent Then, $Ba$ exists, with $Ba = \bigwedge \{ (\neg D)^{n}a: n \in \mathbb{N} \}$, for any $a \in A$. 
\end{pro}

\begin{proof2}
 Considering the definition of $B$, it is enough to prove, for $a \in A$, 
 (i) $\bigwedge \{ (\neg D)^{n}a: n \in \mathbb{N} \} \leq x$, 
 (ii) $\bigwedge \{ (\neg D)^{n}a: n \in \mathbb{N} \} \vee \neg \bigwedge \{ (\neg D)^{n}a: n \in \mathbb{N} \} = 1$, and  
 (iii) if $b \leq a$ and $b \vee \neg b = 1$, then $b \leq \bigwedge \{ (\neg D)^{n}a: n \in \mathbb{N} \}$. 
 Now, (i) follows, because $a \in \{ (\neg D)^{n}a: n \in \mathbb{N} \}$, as $a = (\neg D)^{0}a$. 
 Regarding (ii) and using Lemma \ref{L-D}, it is enough to prove that 
 $\neg D {\color{red}(} \bigwedge \{ (\neg D)^{n}a: n \in \mathbb{N} \} {\color{red})} = \bigwedge \{ (\neg D)^{n}a: n \in \mathbb{N} \}$. 
 As it is always the case, for any $b \in A$, that $\neg Db \leq b$, it suffices to prove that 
 $\bigwedge \{ (\neg D)^{n}a: n \in \mathbb{N} \} \leq \neg D {\color{red}(} \bigwedge \{ (\neg D)^{n}a: n \in \mathbb{N} \} {\color{red})}$.  
 Now, using both properties of Lemma \ref{DM}, we have that the right hand side of the just given inequality is equal to 
 $\bigwedge \{ (\neg D)^{n}a: n \in \mathbb{N}^{+} \}$. 
 It is clear that $\bigwedge \{ (\neg D)^{n}a: n \in \mathbb{N} \} \leq \bigwedge \{ (\neg D)^{n}a: n \in \mathbb{N}^{+} \}$, 
 because $\bigwedge \{ (\neg D)^{n}a: n \in \mathbb{N} \} \leq (\neg D)^{m}a$, for $m \in \mathbb{N}^{+}$.
 Regarding (iii), suppose (iv) $b \leq a$ and $b \vee \neg b = 1$, the last of which implies, by Lemma \ref{L-D}, that (v) $\neg Db = b$. 
 In order to prove that $b \leq \bigwedge \{ (\neg D)^{n}a: n \in \mathbb{N} \}$, it is enough to prove that 
 $b \leq (\neg D)^{n}a$, for all $n \in \mathbb{N}$, which easily follows by induction, as $b \leq a = (\neg D)^{0}a$, by (iv), and 
 supposing that $b \leq (\neg D)^{n}a$, it follows, using Lemma \ref{NDM}, that $\neg Db \leq (\neg D)^{n+1}a$, and, using (v), 
 $b \leq (\neg D)^{n+1}a$. 
\end{proof2}

In Example \ref{ByesDno} we saw that the existence of $B$ in a residuated lattice does not force the existence of $D$. 
Now, let us see that operation $\Delta$ is stronger than $B$ in this respect.

\begin{pro} \label{DeltagivesD}
Let ${\bf{A}} \in \mathbb{RL}$. 
If $\Delta$ exists in {\bf{A}}, then also $D$ exists in {\bf{A}}, with $D = \neg \Delta$ and $\Delta = \neg D$.
\end{pro}

\begin{proof2}
Suppose {\bf{A}} is a residuated lattice where $\Delta$ exists. 
Then, also $\neg \Delta$ exists. 
We have to prove that $Da = \neg \Delta a$, for any $a \in A$. 
We have that $a \vee \neg \Delta a = 1$, as $\Delta a \vee \neg \Delta a = 1$ and $\Delta a \leq a$. 
Now, suppose $a \vee b = 1$. 
Then, $\Delta(a \vee b) = \Delta 1 = 1$. 
It is also the case that $\Delta(a \vee b)=\Delta a \vee \Delta b$. 
So, $\Delta a \vee \Delta b = 1$. 
Using Lemma \ref{GRL}(vi), $\neg \Delta a \leq \Delta b$. 
Moreover, $\Delta b \leq b$. 
So, $\neg \Delta a \leq b$.

Let us also see that $\Delta a = \neg Da$, for all $a$. 
From the first part it follows that $\neg Da = \neg \neg \Delta a$. 
Now, from {\bf{($\Delta$E2)}}, using Lemma \ref{GRL}(vi), it follows that $\neg \neg \Delta a \leq \Delta a$.  
And using Lemma \ref{GRL}(vii) we have that $\Delta a \leq \neg \neg \Delta a$. 
\end{proof2}

\begin{rem}
The reciprocal of Proposition \ref{DeltagivesD} is not the case, 
as $D$ exists in Example \ref{5}, but $\Delta$ does not. 
\end{rem}

\begin{cor}
 Let ${\bf{A}} \in \mathbb{RL}$. 
 If $\Delta$ exists in ${\bf{A}}$, then also $B$ and $D$ exist, and we have $\Delta = B = DD = \neg D$. 
\end{cor}

\begin{proof2}
 Considering Propositions \ref{DimB} and \ref{DeltagivesD}, it is enough to prove that $\Delta = DD$, 
 which follows from $\neg \Delta \neg \Delta = \Delta$. 
 As for any $a \in A$, $\Delta a$ is Boolean due to {\bf{($\Delta$E2)}}, using Lemma \ref{DB} it is enough to check that $\neg \neg \Delta = \Delta$, 
 which follows again from {\bf{($\Delta$E2)}} and Lemma \ref{GRL}(vi).  
\end{proof2}

\section{The logic FL$^B_{ew}$}

In this section we introduce an expansion of FL$_{ew}$ with a unary connective $B$, 
whose intended algebraic semantics is the variety of $\mathbb{RL}^B$-algebras studied in Section 3. 

Indeed, we define FL$^B_{ew}$ as the expansion of FL$_{ew}$  with the following axiom schemas: 

\medskip

\begin{tabular}{ll}
{\bf{(B1)}} & $B\varphi \to \varphi$, \\
{\bf{(B2)}} & $B\varphi \vee \neg B\varphi$, \\
{\bf{(B3)}} & $B(\varphi \vee \neg \varphi) \to (B\varphi \vee \neg \varphi)$, \\ 
{\bf{(B4)}} & $B(\varphi \to \psi) \to (B\varphi \to B \psi)$.  
\end{tabular}

\medskip

\noindent and the following additional rule: 

\medskip

\begin{tabular}{ll}
{\bf{(B)}} & From $\varphi$ derive  $B\varphi$. 
\end{tabular}

\medskip

\noindent We denote (finitary) derivability in FL$^B_{ew}$ by $\vdash$. 

Note that we have the following facts. 

\begin{lem}\label{LFLB}
 (i) $\vdash B\varphi \to BB\varphi$ and
 
 (ii) $B\varphi \to \psi \vdash B\varphi \to B\psi$. 
\end{lem}

\begin{proof2}
For (i) check the following derivation:

\vspace{5pt}

\begin{tabular}{l l l}
 1. & $B\varphi \vee \neg B\varphi$ & {\bf{(B2)}} \\ 
 
 2. & $B(B\varphi \vee \neg B\varphi)$ & 1, rule {\bf{(B)}} \\ 
 
 3. & $BB\varphi \vee \neg B\varphi$ & {\bf{(B3)}}, 2, mp \\ 
 
 4. & $BB\varphi \to (B\varphi \to BB\varphi)$ & FL$_{ew}$ \\  
 
 5. & $\neg B \varphi \to (B \varphi \to BB \varphi)$ & FL$_{ew}$ \\  
 
 6. & $B\varphi \to BB\varphi$ & 3, 4, 5, FL$_{ew}$ \\  
 \end{tabular}

 \vspace{5pt}
 
\noindent (ii) follows easily using (i). 
\end{proof2}

Clearly, FL$^B_{ew}$ is a {\em Rasiowa implicative} logic (cf.\ \cite{Ras}). 
Then, it follows that it is algebraizable in the sense of Blok and Pigozzi \cite{BP89}.  
It is also straightforward to check that the variety $\mathbb{RL}^B$ is its equivalent algebraic semantics. 
Algebraizability immediately implies strong completeness of FL$^B_{ew}$ with respect to $\mathbb{RL}^B$. 

\begin{teo}
 For every set $\Gamma\cup \{\varphi\}$ of formulas, $\Gamma \vdash \varphi$ iff 
 for every $\alg{A} \in \mathbb{RL}^B$ and every $\alg{A}$-evaluation $e$, $e(\varphi) = 1$, 
 whenever $e[\Gamma] \subseteq \{1\}$.
\end{teo}

In FL$^B_{ew}$ the usual form of the deduction theorem does not hold. 
Indeed, $\varphi \vdash B\varphi$, but $\nvdash \varphi \to B\varphi$, 
Which fails in the three-element G\"{o}del algebra $[0, \frac{1}{2}, 1]_G$, where $B1 = 1$ and $B\frac{1}{2} = B0 = 0$: 
for any evaluation $e$ in this algebra, if $e(\varphi) = 1$, then $e(B\varphi) = 1$, 
but for $e(\varphi) = \frac{1}{2}$ we have $e(B\varphi) = 0$, and thus $e( \varphi \to B\varphi) = 0$.   

Actually, FL$^B_{ew}$ enjoys the same form of deduction theorem holding for logics with the $\Delta$ operator (cf. \cite[Proposition 2.2]{Ba}).  

\begin{teo} \label{Deduction}
$\Gamma , \varphi \vdash \psi$ iff $\Gamma \vdash B\varphi \rightarrow \psi$.
\end{teo}

\begin{proof2}
$\Rightarrow$)
We prove by induction on every formula $\chi_{i}$ ($1 \leq i \leq n$) of the given derivation of $\psi$ from $\Gamma \cup \{ \varphi \}$ 
that $\Gamma \vdash B\varphi \to \chi_{i}$. 
If $\chi_{i} = \varphi$, then the result follows due to axiom schema {\bf{(B1)}}. 
If $\chi_{i}$ belongs to $\Gamma$ or is an instance of an axiom, 
then the result follows using \emph{modus ponens} and the derivability of the schema $\chi_{i} \to (B\varphi \to \chi_{i})$. 
If $\chi_{i}$ comes by application of \emph{modus ponens} on previous formulas in the derivation, then the result follows, 
because from $B\varphi \to \chi_{k}$ and $B\varphi \to (\chi_{k} \to \chi_{i})$ we may derive 
$(B\varphi \conj B\varphi) \to (\chi_{k} \conj (\chi_{k} \to \chi_{i}))$ and then also $B\varphi \to \chi_{i}$, 
using transitivity of $\to$ applied to the derivable formulas $B\varphi \to (B\varphi \conj B\varphi)$ and 
$(\chi_{k} \conj (\chi_{k} \to \chi_{i})) \to \chi_{i}$. 
Finally, if $\chi_{i} = B\chi_k$ comes using rule {\bf{(B)}} from formula $\chi_k$, 
then from $B \varphi \to \chi_k$ we may derive $B \varphi \to B\chi_k$ using Lemma \ref{LFLB}(ii). 

$\Leftarrow$) To the derivation given by the hypothesis add a step with $\varphi$. 
In the next step put $B\varphi$, which follows from the previous formula using rule {\bf{(B)}}. 
Finally, derive $\psi$ using \emph{modus ponens}.
\end{proof2}

Thanks to this $B$-deduction theorem, the logic FL$^B_{ew}$ has the following property: 
if we expand FL$^B_{ew}$ with any further rule $\varphi_{1}$, \textellipsis , $\varphi_{n}/ \varphi$, 
then it is possible to dispose of the rule just adding the axiom $(B\varphi_{1} \wedge \cdots \wedge B\varphi_{n}) \to \varphi$. 
This property also holds for the logics FL$_{ew}^\Delta$ and FL$_{ew}^D$.

\begin{pro}
{\rm FL$^B_{ew}$} is a conservative expansion of {\rm FL$_{ew}$}.
\end{pro}

\begin{proof2}
Use Proposition \ref{finB} and the Finite Model Property of FL$_{ew}$ (see \cite{Ok}).
\end{proof2}

One could analogously define the expansion of MTL (which is in turn the extension of FL$_{ew}$ with the pre-linearity axiom 
$(\varphi \to \psi) \lor (\psi \to \varphi)$) with $B$, with the same additional axioms and rule, yielding the logic MTL$^B$, 
which is algebraizable and strongly complete with respect to the variety $\mathbb{MTL}^B$ of MTL$^B$-algebras. 
However, unlike the case of expansion with $\Delta$, MTL$^B$ is not a semilinear logic, that is, 
it is not complete with respect to the class of MTL$^B$-chains. 
The reason is that the $\lor$-form of rule {\bf{(B)}}, 
``from $\psi \lor \varphi$ derive $\psi \lor B\varphi$'', is not derivable in  MTL$^B$. 
Indeed, taking the coatoms $s$ and $t$ in the G\"odel algebra of Example \ref{5}, 
it is clear that $s \vee t = 1$, while $s \vee Bt = s \vee 0 = s$.  

As a final result, we can show that FL$^B_{ew}$ inherits from FL$_{ew}$ the Finite Model Property (FMP). 
Before proving this, we introduce some preliminary notation. 

For a logic $L \in \{$ FL$_{ew}$ or FL$_{ew}^B\}$, 
let us denote by $Fm(L, Var)$ the set of $L$-formulas built from a set $Var$ of propositional variables. 
Now let us define the enlarged set of propositional variables $Var^* = Var \cup \{``B\varphi" \mid B\varphi \in Fm($FL$_{ew}^B, Var)\}$, 
where $``B\varphi"$ is intended to denote a fresh propositional variable, one for each formula $B\varphi \in Fm($FL$_{ew}^B, Var)$. 
Then, we can define a one-to-one translation of every formula  $\varphi \in Fm($FL$_{ew}^B, Var)$ 
into a formula $\varphi^* \in Fm($FL$_{ew}, Var^*)$, by just inductively defining: 

\medskip

\noindent 
-  $0^* = 0$,  \\
- if $\varphi = p  \in Var$, then $\varphi^* = p$, \\
- if $\varphi = B\psi$, then $\varphi^* = ``B\psi"$, \\
- if $\varphi = \psi \odot \chi$, then $\varphi^* = \psi^* \odot \chi^*$, for $\odot \in \{\land, \lor, \&, \to \}$. 

\medskip

\noindent If $\Gamma$ is a set of formulas, we write $\Gamma^* = \{\varphi^* \mid \varphi \in \Gamma\}$. 
Note that for any $\psi \in Fm($FL$_{ew}, Var^*)$, there is a formula $\varphi \in Fm($FL$_{ew}^B, Var)$ such that $\varphi^* = \psi$. 

Moreover, we need the following result that will allow us to reduce proofs in FL$_{ew}^B$ to proofs in FL$_{ew}$. 

\begin{lem} \label{reduce} Let $T$ be the set of all instances of axioms of {\rm FL$_{ew}^B$}. 
For each set {\rm $\Gamma \cup \{\varphi\} \subseteq Fm($FL$_{ew}^B, Var)$}, it holds that 
\begin{center}
$\Gamma \vdash_{FL_{ew}^B} \varphi$ iff $\Gamma^* \cup Cg^* \cup  T^*  \vdash_{FL_{ew}} \varphi^*$, 
\end{center}
where $Cg = \{ B\varphi \leftrightarrow B\psi \mid \Gamma  \vdash_{FL_{ew}^B} \varphi \leftrightarrow \psi\}$.
\end{lem} 

\noindent The proof is quite straightforward and analogous to those of similar results 
that can be found in the literature in slightly different contexts. 

\begin{teo} {\rm FL$^B_{ew}$} enjoys the FMP, that is, if $\Gamma \not\vdash_{FL_{ew}^B} \varphi$, 
then there is a finite ${\bf A} \in \mathbb{RL}^B$ and an ${\bf A}$-evaluation $e$ such that 
$e(\Gamma) = 1$ and $e(\varphi) < 1$. 
\end{teo}

\begin{proof2} If  $\Gamma \not\vdash_{FL_{ew}^B} \varphi$, by Lemma \ref{reduce}, 
it holds that $\Gamma^* \cup Cg^* \cup  T^*  \not\vdash_{FL_{ew}} \varphi^*$, and by strong completeness and FMP of FL$_{ew}$, 
there is a finite algebra ${\bf C} \in \mathbb{RL}$ and {\bf C}-evaluation $v$ such that 
$v(\Gamma^* \cup Cg^* \cup  T^*) = 1$ and $v(\varphi^*) < 1$. 
Then, the result will follow from the following facts: \\

\noindent {\bf \em Claim 1}: 
$G = \{ v(``B\varphi \text{''}) \mid B\varphi \in Fm($FL$_{ew}^B, Var) \}$ is a set of Boolean elements of $\bf C$. ``You \emph{were} a little grave,''

\begin{proofclaim} It is enough to check that $v((B\varphi)^*) \lor \neg v((B\varphi)^*) = v((B\varphi)^* \lor \neg (B\varphi)^*) = v((B\varphi \lor \neg B\varphi)^*) = 1$, 
where the latter holds because $B\varphi \lor \neg B\varphi$ is the axiom {\bf{(B2)}} of  FL$_{ew}^B$.
\end{proofclaim}

\noindent {\bf \em Claim 2}: 
Let $\bf A$ be the $\mathbb{RL}$-algebra generated by the set 
$X = \{ v(\varphi) \mid \varphi \in Fm($FL$_{ew}, Var^*) \}$, which is finite since ${\bf A}$ is a subalgebra of ${\bf C}$. 
Then, $B$ exists in $\bf A$ and $B({\bf A}) = G$. 
Therefore, $\bf A$ is indeed an $\mathbb{RL}^B$-algebra. 

\begin{proofclaim}  That $\bf A$ is finite is obvious, and thus, by Proposition 5, $B$ exists. 
On the other hand, the elements of $G$ keep being Boolean in $\bf A$. 
Hence, the only missing thing to check is that any Boolean element of $\bf A$ already belongs to $G$. 
This is also clear since Boolean elements are closed by propositional combinations with connectives. 
\end{proofclaim}

\noindent  {\bf \em Claim 3}: Let us define the $\bf A$-evaluation (taking $\bf A$ as $ \mathbb{RL}^B$-algebra) 
$e: Var \to A$ defined by $e(p) = v(p)$. 
Then, for any $\varphi$,  $e(\varphi) = v(\varphi^*)$, in particular, $e(B\varphi) = v(``B \varphi")$. 

\begin{proofclaim} We prove that $e(\varphi) = v(\varphi^*)$ by structural induction. 

\begin{itemize}
\item[-] if $\varphi$ is a propositional variable, it holds by construction
\item[-] if $\varphi =  \psi \odot \chi$ for $\odot \in \{\land, \lor, \&, \to \}$,  
by induction hypothesis we have $e(\psi) = v(\psi^*)$ and $e(\chi) = v(\chi^*)$, and hence  
$e(\varphi) = e(\psi \odot \psi) = e(\psi) \odot e(\chi) = v(\psi^*) \odot v(\chi^*) = v(\psi^* \odot \chi^*) =  v((\psi \odot \chi)^*)  = v(\varphi^*)$. 

\item[-] If $\varphi =  B\psi$, then we have to prove that $v(``B \psi") =  B(e(\psi))$, the latter being equal to $ e(B\psi)$ by definition. 
Therefore, we have to prove in turn that the three defining conditions {\bf{(BE1)}}, {\bf{(BE2)}}, and {\bf{(BI)}} are satisfied by 
$v(``B \psi") = v((B\psi)^*)$ to be the greatest Boolean below $e(\psi)$, assuming by induction that $v(\psi^*) = e(\psi)$. 
\begin{itemize}
\item[{\bf{(BE1)}}] Since $B\psi \to \psi$ is axiom {\bf{(BE1)}} of $FL_{ew}^B$, 
we have that $1 = v((B\psi \to \psi)^*) = v((B\psi)^*) \to v(\psi^*) =  v((B\psi)^*) \to e(\psi)$. Hence, $v((B\psi)^*) \leq e(\psi)$.  

\item[{\bf{(BE2)}}] is clear from Claim 1. 

\item [{\bf{(BI)}}] We have to check that if $b \in B(A)$ is such that $b \leq e(\psi) = v(\psi^*)$, then $b \leq v((B\psi)^*)$.  
If $b \in B(A)$, by construction of $\bf A$, then there exists a formula $\chi$ such that $b = v((B\chi)^*)$. 
On the other hand, by (ii) of Lemma 12, we know that  $B\chi \to \psi, B\chi \lor \neg B\chi \vdash B\chi \to B\psi$. 
Thus, we also know that if $v((B\chi)^*) \leq v(\psi)^*$ and $v((B\chi)^*) \lor \neg v((B\chi)^*) = 1$, then $v((B\chi)^* \leq v((B\psi)^*)$. 
Now, the two conditions are satisfied, hence we have $b = v((B\chi)^* \leq v((B\psi)^*)$. 
\end{itemize}
\end{itemize}
This closes the proof of Claim 3. 
\end{proofclaim}
Finally, from these claims it readily follows that $e(\Gamma) = v(\Gamma^*) = 1$ and $e(\varphi) = v(\varphi^*) < 1$, as required. 
\end{proof2}

\section{Conclusions and dedication}

In this paper we have considered the expansion of FL$_{ew}$ with the operator $B$, 
that in algebraic terms provides the greatest Boolean below a given element of a residuated lattice. 
Among other things, we have axiomatized it and shown that the resulting logic is a conservative expansion enjoying the Finite Model Property. 
The axioms for $B$ turn out to be very close to those of the Monteiro-Baaz $\Delta$ operator, 
in fact only one axiom is a weaker version of the one for $\Delta$. 
Even if the properties are very similar, that small difference causes, e.g. that in the context of MTL, 
the expansion with $B$ is not any longer a semilinear logic, in contrast to the expansion with $\Delta$. 

As a matter of fact, we have chosen this topic for our humble contribution to honour the memory of our beloved and late friend Franco Montagna, 
because it was suggested by Franco to the first author during the preparation of their joint manuscript \cite{AEM},  together with Amidei,  
where they study the expansion of FL$_{ew}$ and other substructural logics with $\Delta$.    

\section*{Acknowledgments and Compliance with Ethical Standards} The authors are thankful to an anonymous reviewer for his/her comments that have helped to improve the final layout of this paper. The authors have been funded by the EU H2020-MSCA-RISE-2015 project 689176--SYSMICS. 
Esteva and Godo  have been also funded by the FEDER/MINECO Spanish project TIN2015-71799-C2-1-P. 
The authors declare that they have no conflict of interest. 
This article does not contain any studies with human participants or animals performed by any of the authors. 
 
\bibliographystyle{plain}

\begin{thebibliography}{}


\bibitem{AEM} Amidei, Jacopo, Ertola-Biraben, Rodolfo C., and Montagna, Franco. 
Conservative Expansions of Substructural Logics. (submitted). 
Preprint available as CLE e-Prints, Vol. 16(2), 2016. 
(\url{http://www.cle.unicamp.br/e-prints/vol_16,n_2,2016.html})

\bibitem{Ba} Baaz, Matthias. Infinite-valued G\"{o}del Logics with 0-1-Projections and Relativizations. 
In \emph{G\"{O}DEL'96 - Logical Foundations of Mathematics, Computer Science and Physics; Lecture Notes in Logic 6}, 
1996, P. H\'{a}jek, Ed., Springer-Verlag, pp. 23-33.

\bibitem{BP89}
Blok, Willem J. and Pigozzi, Don L.
\newblock {\em Algebraizable Logics}, volume 396 of {\em Memoirs of the American Mathematical Society}.
\newblock American Mathematical Society, Providence, RI, 1989.

\bibitem{CC} Caicedo, Xavier and Cignoli, Roberto. An Algebraic Approach to Intuitionistic Connectives. 
\emph{The Journal of Symbolic Logic}, 66:4, 2001, pp. 1620-1636.

\bibitem{CEB} Castiglioni, Jos\'{e} L. and Ertola Biraben, Rodolfo C. Modal operators in meet-complemented lattices. 
(submitted). Preprint available as arXiv:1603.02489 [math.LO] (\url{http://arxiv.org/abs/1603.02489})

\bibitem{Ci} Cignoli, Roberto. Boolean Elements in Lukasiewicz Algebras. I. 
\emph{Proc. Japan Acad.}, 41, 1965, pp. 670-675.

\bibitem{CM} Cignoli, Roberto and Monteiro, A. Boolean Elements in Lukasiewicz Algebras. II. 
\emph{Proc. Japan Acad.}, 41, 1965, pp. 676-680.

\bibitem{Handbook} Cintula, Petr, H\'ajek, Petr, and Noguera, Carles (eds), 
Handbook of Mathematical Fuzzy Logic, 2 vols, Studies in Logic, vols. 37-38, College Publications, London, 2011.

\bibitem{GJKO} Galatos, Nikolaos, Jipsen, Peter, Kowalski, Tomasz, and Ono, Hiroakira. 
\emph{Residuated Lattices: An Algebraic Glimpse at Substructural Logics},  
Studies in Logic and the Foundations of Mathematics, Vol. 151, Elsevier, New York, 2007. 

\bibitem{Ha} H\'{a}jek, Petr. \emph{Metamathematics of Fuzzy Logic}, Springer, 1998. 

\bibitem{Moi1} Moisil, Grigore. Logique modale. \emph{Disquisitiones Mathematicae et Physica}, 2, 1942, pp. 3-98. 

\bibitem{Moi2} Moisil, Grigore. \emph{Essais sur les logiques non chrysippiennes}, 
Editions de l'Academie de la Republique Socialiste de Roumanie, Bucarest, 1972. 

\bibitem{Mon} Monteiro, Ant\'{o}nio. Sur les alg\`{e}bres de Heyting sym\'{e}triques. 
\emph{Portugaliae Mathematica}, 39:1-4, 1980, pp. 1-237. 

\bibitem{No07} Noguera, Carles. Algebraic study of axiomatic extensions of triangular norm based Fuzzy Logics. 
\emph{Monografies de l'Institut d'Investigaci\'o en Intel.lig\`encia artificial (IIIA-CSIC)}, Vol.32,  2007. 

\bibitem{Ok} Okada, Mitsuhiro and Terui, Kazushige. The Finite Model Property for Various Fragments of Intuitionistic Linear Logic. 
\emph{The Journal of Symbolic Logic}, 64:2, 1999, pp. 790-802.

\bibitem{Rau} Rauszer, Cecylia. Semi-Boolean algebras and their applications to intuitionistic logic with dual opertions. 
\emph{Fundamenta Mathematicae}, 83, 1974, pp. 219-249. 

\bibitem{Ras} Rasiowa, Helena.
\newblock {\em An Algebraic Approach to Non-Classical Logics}.
\newblock North-Holland, Amsterdam, 1974.

\bibitem{RZ} Reyes, Gonzalo E. and Zolfaghari, Houman. Bi-Heyting Algebras, Toposes and Modalities. 
\emph{Journal of Philosophical Logic}, 25:1, 1996, pp. 25-43.

\bibitem{Ri} Ribenboim, Paulo. Characterization of the sup-complement in a distributive lattice with last element. 
\emph{Summa Brasiliensis Mathematicae}, 2:4, 1949, pp. 43-49.

\bibitem{Sk1} Skolem, Thoralf. Untersuchungen \"{u}ber die Axiome des Klassenkalk\"{u}ls und \"{u}ber
Produktations- und Summationsprobleme, welche gewisse Klassen von Aussagen betreffen,
\emph{Skrifter uitgit av Videnskapsselskapet i Kristiania, I}, Matematisk-naturvidenskabelig klasse, 3, 1919, pp. 1-37.

\bibitem{Sk2} Skolem, Thoralf.
\emph{Selected Works in Logic}. Edited by Jens Erik Fenstad, Universitetforlaget, Oslo, 1970.

\bibitem{Sz}
Sz\'asz, Gabor. 
\newblock {\em Th\'eorie des treillis},  {\em Monographies universitaires de Math\'ematiques}.
\newblock Dunod, Paris, 1971.

\end{thebibliography}

\end{document}